\documentclass[11pt]{amsart}

\usepackage{amsmath}
\usepackage{amssymb}
\usepackage{amscd}
\usepackage{graphics}
\usepackage{hyperref}
\usepackage{tikz}

\newtheorem{theorem}{Theorem}[section]
\newtheorem{thm}[theorem]{Theorem}
\newtheorem{cor}[theorem]{Corollary}
\newtheorem{lem}[theorem]{Lemma}
\newtheorem{prop}[theorem]{Proposition}

\theoremstyle{definition}
\newtheorem{defn}[theorem]{Definition}

\newtheorem{conv}[theorem]{Convention}

\newtheorem{conj}[theorem]{Conjecture}

\theoremstyle{remark}

\newcommand{\mbb}{\mathbb}

\newcommand{\ZZ}{\mbb{Z}}
\newcommand{\CC}{\mbb{C}}

\newcommand{\AAA}{\mbb{A}}
\newcommand{\PP}{\mbb{P}}

\newcommand{\HH}{\mbb{H}}
\newcommand{\mc}{\mathcal}

\newcommand{\mcN}{\mc{N}}

\newcommand{\mcX}{\mc{X}}

\newcommand{\mfm}{\mathfrak{m}}

\newcommand{\OO}{\mc{O}}

\newcommand{\wht}{\widehat}
\newcommand{\whts}{\wht{s}}
\newcommand{\whtt}{\wht{t}}
\newcommand{\whtOO}{\wht{\OO}}
\newcommand{\whtr}{\wht{r}}

\newcommand{\SP}{\text{Spec }}

\newcommand{\p}{\mathbb{P}^{1}}

\newsavebox{\sembox}
\newlength{\semwidth}
\newlength{\boxwidth}

\newsavebox{\semrbox}
\newlength{\semrwidth}
\newlength{\boxrwidth}

\title
{Weak approximation for cubic hypersurfaces}

\author[Tian]{Zhiyu Tian}
\address{
Department of Mathematics 253-37\\
California Institute of Technology \\
Pasadena, CA, 91125}
\email{tian@caltech.edu}

\date{\today}
\begin{document}


\begin{abstract}
We prove weak approximation for smooth cubic hypersurfaces of dimension at least $2$ defined over the function field of a complex curve.
\end{abstract}


\maketitle


\section{introduction}
Given an algebraic variety $X$ over a number field or function field $F$, a natural question is whether the set of rational points $X(F)$ is non-empty. If it is non-empty, then how many rational points are there? In particular, are they Zariski dense? Do they satisfy \emph{weak approximation}? 

In this article, we address the weak approximation question for cubic hypersurfaces defined over the function field of a complex curve.

Smooth cubic hypersurfaces of dimension at least $2$ belong to the class of ``rationally connected varieties". Roughly speaking, a smooth projective variety $X$ is \emph{rationally connected} if for any two general points $x$ and $y$, there is a rational curve connecting them. For more precise definition and properties of rationally connected varieties, see \cite{Kollar96}, Chap. IV.

After the pioneering work of Graber-Harris-Starr \cite{GHS03}, which established the existence of rational points for rationally connected varieties defined over the function field of a curve, there has been much research centered around the arithmetic of such varieties, see, e.g. \cite{HT06}, \cite{WAhypersurface}, \cite{BadReduction}, \cite{WAIsotrivial}, \cite{Knecht}, \cite{XuWA}.

In particular, Hassett and Tschinkel \cite{HT06} conjectured that weak approximation holds for all smooth projective rationally connected varieties defined over the function field of a complex curve.

We first state the conjecture in the arithmetic form. Let $B$ be a smooth projective connected complex curve. For any closed point $b \in B$, denote by $\whtOO_{B, b}$ the completion of the local ring at the point $b$ and $\text{Frac}\whtOO_{B, b}$ the corresponding fraction field. Let
$$
\AAA=\prod_{b \in B}^{\circ} \text{Frac}\whtOO_{B, b}
$$
be the ad{\'e}les over the function field $\CC(B)$, where all but finitely many of the factors in the product are in $\whtOO_{B, b}$. Finally let $\mcX_\eta$ be a smooth rationally connected variety defined over the function field $\CC(B)$. Then the weak approximation conjecture can be formulated as saying that the set of rational points $\mcX_\eta (\CC(B))$ is dense in $\mcX_\eta(\AAA)$.

For our purpose, it is more useful to formulate the conjecture in the following geometric form.

\begin{conj}[\cite{HT06}]
 Let $\pi:\mcX \rightarrow B$ be a flat surjective morphism from a projective variety to a smooth projective curve such that a general fiber is smooth and rationally connected (such a map $\mcX \rightarrow B$ is called a \emph{model} of the generic fiber).
Then the morphism $\pi$ satisfies
\emph{weak approximation}. That is, for every finite sequence
$(b_1,\dots,b_m)$ of distinct closed points of $B$, for every sequence
$(\whts_1,\dots,\whts_m)$ of formal power series sections of $\pi$
over $b_i$, and for every positive integer $N$, there exists a regular
section $s$ of $\pi$ which is congruent to
$\whts_i$ modulo $\mfm_{B,b_i}^{N}$ for every $i=1,\dots,m$.
\end{conj}

For the equivalence of the two formulations, see section 1 of the survey article \cite{WASurvey}, which provides a nice introduction and summary of known results of weak approximation in the function field case (as of 2008).

Some special cases of the conjecture are known, e.g.
\begin{itemize}
\item $\PP^n$, conic bundles over $\PP^1$, del Pezzo surfaces of degree at least $4$, \cite{CTGilleWA},
\item low degree complete intersections of degree $(d_1, \ldots, d_c)$ such that $\sum d_i^2 \leq n+1$, \cite{dHS}, \cite{WASurvey},
\item smooth cubic hypersurfaces in $\PP^n, n\geq 6$ \cite{WAhypersurface},
\item isotrivial families \cite{WAIsotrivial},
\item at places of good reduction (for any family) \cite{HT06},
\item a general family of del Pezzo surfaces of degree at most $3$, \cite{BadReduction},\cite{Knecht}, \cite{XuWA}, and
\item a smooth hypersurface with square-free discriminant \cite{WAhypersurface}.
\end{itemize}

We notice an interesting difference between the cases completely understood and the other cases. Namely, the former cases are proved by studying the global geometry over the function field while the latter cases are proved by studying the local singular fibers and the singularities of the total space.

In some sense, our paper is a combination of the two approaches. The main theorem is the following.
\begin{thm}\label{main}
Let $X$ be a smooth cubic hypersurface of dimension at least $2$ defined over the function field of a complex curve. Then weak approximation holds at all places.
\end{thm}

We conclude this introduction by explaining the idea of the proof. First of all, by choosing a Lefschetz pencil and using standard facts about weak approximation, one reduces to prove weak approximation for cubic surfaces.

There are two new ingredients in the proof. The first one is local. The observation is that when the central fiber is a cone over an irreducible plane cubic curve or non-normal (i.e. the worst degenerate case for families of cubic surfaces, see \ref{Corti}), one can make a ramified base change and a birational modification so that the new central fiber has at worst du Val singularities. One just needs to keep track of the Galois action to get back to the original family. The singularities have been greatly improved during this process.

The second new idea is global and geometric. Given a local formal section that we want to approximate, say $\whts$,  we choose a section $s$, whose restriction to the formal neighborhood  gives a formal section $\whts_0$. Then the two formal sections determine a unique line $\wht{L}$, which intersects the family at a third local formal section $\whts'$. One can find a line $L$ defined over the function field, which contains the rational point corresponding to the section $s$, and approximate $\wht{L}$ to order $N$ (i.e. weak approximation for the space of lines containing $s$). The line $L$ (generally speaking) intersects the family at $s$ and a degree $2$ multisection $\sigma$.

The next step is to deform $\sigma$ so that it approximates the formal sections $\whts_0$ and $\whts'$. This is equivalent to a special case of weak approximation after a degree $2$ base change. However, the multisection $\sigma$ already approximates $\whts'$ (even though we have no control on this formal section) to order $N$. Thus one only needs to approximate $\whts_0$, which comes from the section and can be carefully chosen to lie in the smooth locus of the fibration (Lemma \ref{lem:smoothlocus}) so that weak approximation is possible (by the local approach, i.e. the study of singular fibers).

Once we approximate the formal sections $\whts'$ and $\whts_0$, we take the line spanned by the degree $2$ multisection and the third intersection point with the cubic surface is what we need.

Both Swinnerton-Dyer \cite{SDCubic} and Madore \cite{MadoreCubic} have used the composition law of the cubic surface to study weak approximation on cubic surfaces. However, the basic strategy seems quite different. Their idea is to use a unirational parameterization to approximate $v$-adic points, which only works for places of good reduction. Our approach is to use the composition law to reduce the problem to a special and easier case, which can be proved via the deformation technique of Koll\'ar-Miyaoka-Mori \cite{KMM92RC}.

\textbf{Acknowledgments:} I would like to thank Tom Graber for many helpful discussions and for saving me from making many false statements, Chenyang Xu and Runpu Zong for their interest in the project, Letao Zhang for her help with the preparation of the manuscript, and the referees for so many helpful suggestions that have greatly improved the paper, both in its content and in its way of presentation.

\section{Preliminaries}

\subsection{Everything with a cyclic group action}
In this subsection we collect some useful results from \cite{WAIsotrivial}. Let $k$ be any field and $G$ a cyclic group of order $l$ such that $l$ is invertible in $k$.

First, we are concerned with the following infinitesimal lifting problem. Let $S$ and $R$ be $k$-algebras with a $G$-action and $f: S \rightarrow R$ be an algebra homomorphism compatible with the action. Let $A$ be an Artinian $k$-algebra with a $G$-action, $I \subset A$ an invariant ideal such that $I^2=0$. Consider the following commutative diagram, where $p$ is a $G$-equivariant $k$-algebra homomorphism.
\[
\begin{CD}
S@>{f}>>R\\
  @VVV		@VV{p}V\\
A@>\pi>>A/I@>>>0\\
\end{CD}
\]

We want to know when one can find a $G$-equivariant lifting $h:R\rightarrow A$. The following lemma completely answers this question.

\begin{lem}\label{lem-equiv-lifting}
If we can lift the map $p$ to a $k$-algebra homomorphism $h: R\rightarrow A$ such that $\pi\circ h=p$, then we can find an equivariant lifting $\tilde{h}:R \rightarrow A$ with the same property.
\end{lem}

\begin{proof}
For every element $g$ in $G$, define a map $h_g:R\rightarrow A$ by $h_g(r)=g\cdot h(g^{-1}\cdot r)$. This is an $S$-algebra homomorphism and also a lifting of the map $p: R\rightarrow A/I$. The map $h$ is $G$-equivariant if and only if $h_g(r)=h(r)$ for every $g\in G$ and every $r \in R$. The difference of any two such liftings is an element in $Hom(\Omega_{R/S}, I)$, where $\Omega_{R/S}$ is the module of relative differentials. Therefore one has $\theta(g)(r)=h_g(r)-h(r)$ in $Hom(\Omega_{R/S}, I)$. Notice that $Hom(\Omega_{R/S}, I)$ is naturally a $G$-module with the action of $G$ on $Hom(\Omega_{R/S}, I)$ given by
\[
G \times Hom(\Omega_{R/S}, I) \rightarrow Hom(\Omega_{R/S}, I)
\]
\[
(g, \eta) \mapsto g\cdot\eta=(\omega \mapsto g\cdot\eta(g^{-1}\cdot \omega)).
\]

It is easy to check that
\[
\theta(gh)=g\cdot\theta(h)+\theta(g)
\]
Thus $\theta$ defines an element $[\theta]$ in $H^1(G,Hom(\Omega_{R/S}, I))$. The existence of an equivariant lifting is equivalent to the existence of an element $\Theta \in \textsl{Hom}(\Omega_{R/S}, I)$ such that $g \Theta-\Theta=\theta$, i.e, the class defined by $\theta$ is zero in $H^1(G,Hom(\Omega_{R/S}, I))$. Since the characteristic of the field is relatively prime to the order of $G$, all the higher cohomology groups $H^i(G, Hom(\Omega_{R/S}, I)), i \geq 1$ of $G$ vanish (\cite{weibel94}, Proposition 6.1.10, Corollary 6.5.9). The vanishing can be proved by the usual averaging argument.
\end{proof}

\begin{cor}\label{sm-equiv-lifting}
Let $X$ and $Y$ be two $k$-schemes with a $G$-action and $f: X \rightarrow Y$ be a finite type $G$-equivariant morphism. Let $x \in X$ be a fixed point, and $y=f(x)$ (hence also a fixed point). Assume that $f$ is smooth at $x$. Then there exists a $G$-equivariant section $s: \SP \whtOO_{y,Y}\rightarrow X$. In particular, assume that $Y$ is irreducible and the $G$ action on $Y$ is trivial. If there is a fixed point in $X$, then the set the fixed points of $X$ dominates $Y$.
\end{cor}

\begin{proof}
Let $S$ be the local ring at y, and $R$ be the local ring at x. There is an obvious $G$ action on both of these $k$-algebras. We start with the section $s_0:\SP k(y) \rightarrow f^{-1}(y), \SP k(y)\mapsto x$, which is clearly $G$-equivariant. By the smoothness assumption, a section from $\SP(\whtOO_{y,Y}/{\mfm_{y}^n})$ always lifts to a section from $\SP(\whtOO_{y,Y}/{{\mfm_y}^{n+1}})$. Now apply Lemma \ref{lem-equiv-lifting} inductively to finish the proof.
\end{proof}

We also need the following $G$-equivariant smoothing result, which is a slight generalization of the corresponding results in \cite{WAIsotrivial}. 
\begin{lem}\label{lem:EquivSmoothing}
Let $X$ be a smooth quasi-projective rationally connected variety over $\CC$ and $G$ be a cyclic group of order $l$ acting on $X$. Fix an action of $G$ on $\PP^1$ by $z \mapsto \zeta z$, where $\zeta$ is a primitive $l$-th root of unity. Assume that there is a very free rational curve through every point of $X$.
\begin{enumerate}
\item Let $f: \PP^1 \rightarrow X$ be a $G$-equivariant map. Then there exists a $G$-equivariant map $\tilde{f}:\p \rightarrow X$ such that $\tilde{f}(0)=f(0)$, $\tilde{f}(\infty)=f(\infty)$, and $\tilde{f}$ is very free. 
\item Let $f_i: C_i \rightarrow X, 1\leq i \leq n$ be a chain of equivariant maps, i.e. for each $i$, $C_i\cong \p$, and $f_i$ is a $G$-equivariant map such that $f_i(\infty)=f_{i+1}(0)$ for $1 \leq i \leq n-1$. Then there is a $G$-equivariant map $\tilde{f}: \p \rightarrow X$ such that $\tilde{f}(0)=f_1(0)$ and $\tilde{f}(\infty)=f_n(\infty)$.
\end{enumerate}
\end{lem}
Recall that a quasi-projective complex variety is rationally connected if there is a rational curve through a general pair of points. The assumption that there is a very free rational curve through every point of $X$ is saying that $X$ is strongly rationally connected in the sense of Hassett-Tschinkel \cite{BadReduction}.
\begin{proof}
For part (1), we may assume that the equivariant map $f$ is an embedding and $\dim X\geq 3$ by replacing $X$ with $X\times \PP^M$ for some large $M$ and projecting deformations to the first factor $X$. Let $C$ be the image of the morphism $f$. 

We first attach very free curves $C_i$ at general points $p_i \in C$ along general tangent directions at $p_i$. Let $D_1$ be the nodal curve assembled in this way. By Lemma 2.5 \cite{GHS03}, after attaching enough such curves, the twisted normal sheaf $\mcN_{D_1/X}(-0-\infty)$ is globally generated and $H^1(D_1, \mcN_{D_1/X}(-0-\infty)\otimes L_1)=0$, where $L_1$ is any line bundle on $D_1$ which has degree $-l$ on $C$ and $0$ on all the other irreducible components $C_i$. 

We then attach all the curves that are $G$-conjugate to $C_i$'s (we may choose $C_i$'s such that the $G$-orbits do not intersect each other). The new nodal curve is denoted by $D$. The map $D \to X$ is $G$-equivariant. As in the previous paragraph, $D$ has the property that $H^1(D, \mcN_{D/X}(-0-\infty)\otimes L)=0$, where $L$ is any extension of the line bundle $L_1$ on $D$ which has degree $-l$ on $C$ and $0$ on all the attached rational tails. Denote by $R_j$ the rational curve attached to $C$ at the point $p_j$.

We have the following two exact sequences:
\[
0 \to \oplus_j \mcN_{D/X}(-0-\infty)|_{R_j}(-n_j) \to \mcN_{D/X}(-0-\infty) \to \mcN_{D/X}(-0-\infty)|_C \to 0,
\]
\[
0 \to \mcN_{C/X}(-0-\infty) \to \mcN_{D/X}(-0-\infty)|_C \to \oplus_j Q_j \to 0,
\]
where $n_j$'s are the nodal points on $R_j$, and $Q_j$'s are torsion sheaves supported on the points $p_j \in C$. Every sheaf has a natural $G$-action and the $G$-equivariant deformations are given by $G$-invariant sections of $\mcN_{D/X}$. To find a $G$-equivariant deformation smoothing all the nodes of $D$ and fixing $0$ and $\infty$, one just needs to find a $G$-invariant section in $H^0(D, \mcN_{D/X}(-0-\infty))^G$ which, for any $j$, is not mapped to $0$ under the composition of maps
\[
H^0(D, \mcN_{D/X}(-0-\infty)) \to H^0(C, \mcN_{D/X}(-0-\infty)|_C) \to Q_j
\]
for all $j$.

Since $H^1(D, \mcN_{D/X}(-0-\infty)\otimes L)=0$, we also have $$H^1(C, \mcN_{D/X}(-0-\infty)\otimes L\otimes \OO_C)=0.$$ Let $c_1, \ldots, c_l$ be an orbit of the $G$ action on $C$. By the vanishing of $H^1(C, \mcN_{D/X}(-0-\infty)\otimes \OO_C(-x_1-\ldots-x_l))$, the map
\[
H^1(C, \mcN_{D/X}(-0-\infty)\otimes \OO_C(-c_1-\ldots-c_{l-1}))\to  \mcN_{D/X}(-0-\infty)|_{c_l}
\]
is surjective. Thus there is a section of $\mcN_{D/X}(-0-\infty)|_C$ which vanishes on $c_1, \ldots, c_{l-1}$ but not on $c_l$. Then taking the average over $G$ gives a $G$-invariant section of $\mcN_{D/X}(-0-\infty)|_C$ which does not vanish on any of the points $c_1, \ldots, c_l$. In particular, for any $l$ nodes on $C$ which lie in a $G$-orbit, we can find a $G$-invariant section of $\mcN_{D/X}(-0-\infty)$ which does not vanish on them. Then a general $G$-invariant section of $\mcN_{D/X}(-0-\infty)|_C$ does not vanish on any of the nodes $p_i$. 

We have a surjection map $$H^0(D, \mcN_{D/X}(-0-\infty)) \to H^0(C, \mcN_{D/X}(-0-\infty)|_C)$$ and (consequently) $$H^0(D, \mcN_{D/X}(-0-\infty))^G \to H^0(C, \mcN_{D/X}(-0-\infty)|_C)^G$$ is also surjective. So a general $G$-invariant section in $H^0(D, \mcN_{D/X}(-0-\infty))^G$ does not vanish on the nodes. We take the $G$-equivariant deformation given by this section, which necessarily smooths all the nodes of $D$ with $0$ and $\infty$ fixed. A general smoothing is very free since the normal bundle is ample by upper-semicontinuity. 

For the second part, we may assume that all the $f_i$'s are very free by the first part. Let $f$ be the $G$-equivariant map obtained by gluing the $f_i$'s. 

Let $(T, o)$ be a pointed smooth curve with trivial $G$-action, and let $\tilde{\Sigma}$ be $\PP^1\times T$ with the natural diagonal action. There are two $G$-equivariant sections, $s_0=0 \times T, s_\infty=\infty \times T$. Now blow up the point $s_\infty(o)$ and still denote the strict transforms of the two sections by $s_0$ and $s_\infty$. The $G$-action extends to the blow-up. We can make the fiber over $o \in T$ a chain of rational curves with $n$ irreducible components by repeating this operation. Then we get a smooth surface $\Sigma$ with a $G$-action such that the projection to $T$ is $G$-equivariant.  

Let $h_0: s_0 \rightarrow X\times T$ and $h_\infty: s_\infty \rightarrow X \times T$ be $T$-morphisms such that $h_0(s_0)=f_1(0)\times T$ and $h_\infty(s_\infty)=f_n(\infty) \times T$. Consider the relative Hom-scheme $\text{Hom}_T(\Sigma, X\times T, h_0, h_\infty)$ parameterizing $T$-morphisms from $\Sigma$ to $X \times T$ fixing $h_0$ and $h_\infty$. It has a natural $G$ action and the map $\mu: \text{Hom}_T(\Sigma, X\times T, h_0, h_\infty) \rightarrow T$ is $G$-equivariant. Now $\mu$ is smooth at $f$. By Corollary \ref{sm-equiv-lifting}, there is a $G$-equivariant formal section. So there are $G$-equivariant smoothings of the morphism $f$.
\end{proof}

Finally we quote the following theorem from \cite{WAIsotrivial}. For our purpose, we only need to find equivariant rational curves in a few cubic surfaces, which, however, might be singular. Furthermore we want the curve to lie in the smooth locus, so the proof in \cite{WAIsotrivial} does not directly carry over. However it is good to know that such curve exists at least in a desingularization. We will discuss this problem in more detail later in \ref{equiv}.

\begin{thm}\label{thm:equivariant}
Let $X$ be a smooth projective rationally connected variety and let $G$ be a cyclic group of order $l$ with an action on $X$. Choose a primitive $l$-th root of unity $\zeta$ and let $G$ act on $\PP^1$ by $[X_0, X_1] \mapsto [X_0, \zeta X_1]$. Then for each pair $(x, y)$ of fixed points in $X$, there is a $G$-equivariant map $f: \PP^1 \to X$ such that $f(0)=x$ and $f(\infty)=y$.
\end{thm}

\subsection{Iterated blow-up}
Let $\pi: \mcX \to C$ be a flat proper family over a smooth projective connected curve $C$. Let $c \in C$ be a closed point and $\whts_0: \SP \whtOO_{c, C} \to \mcX$ be a formal section. Assume that $\whts_0$ lies in the smooth locus of $\mcX \to C$. The $N$-th iterated blow-up associated to $\whts_0$ is defined inductively as follows.

The $0$-th iterated blow-up $\mcX_0$ is $\mcX$ itself. Assume the $i$-th iterated blow-up $\mcX_i$ has been defined. Let $\whts_i$ be the strict transform of $\whts_0$ in $\mcX_i$. Then $\mcX_{i+1}$ is defined as the blow-up of $\mcX_i$ at the point $\whts_i(c)$.

We remark that if both $\mcX$ and $C$ have a $G$-action such that
\begin{itemize}
\item
the map $\pi: \mcX \to C$ is $G$-equivariant.
\item
The point $c$ is the fixed point of $G$ and $\whts_0$ is $G$-equivariant,
\end{itemize}
then each $\mcX_i$ has a $G$-action such that the natural morphisms $\mcX_{i+1} \to \mcX_i$ and the formal sections $\whts_i$ are $G$-equivariant. In particular, the intersection of $\whts_i$ with the central fiber is a fixed point of $G$.

One can also do this at fibers over a $G$-orbit in $C$, provided the formal sections over these points are conjugate to each other under the $G$-action. Then the iterated blow-up still has a $G$-action and every morphism is compatible with the action.

On $\mcX_N$, the fiber over the point $c$ consists of the strict transform of $\mcX|_c$ and exceptional divisors $E_1, \ldots, E_N$, and
\begin{itemize}
\item
$E_i, i = 1, \ldots, N-1$, is the blowup of $\PP^d$ at $r_i (=\whts_i(c))$, the point where
the proper transform of $\whts_0$ (i.e. $\whts_i$) meets the fiber over $c$ of the $(i-1)$-th iterated blow-up;
\item
$E_N \cong \PP^d$,
\end{itemize}
where $d$ is the dimension of the fiber.

The intersection $E_i \cap E_{i+1}$ is the exceptional divisor $\PP^{d-1}\subset E_{i}$, and a proper transform of a hyperplane in $E_{i+1}$, for $i=0, \ldots, N-1$.

Furthermore, to find a section agreeing with $\whts_0$ to the $N$-th order is the same as finding a section in $\mcX_{N+1}$ intersecting the fiber over $c$ at $E_{N+1}$, or equivalently, a section in $\mcX_N$ which intersects the exceptional divisor $E_N$ at the point $r_N=\whts_N(c)$ (Proposition 11, \cite{HT06}).

\section{Standard models of cubic surfaces over a Dedekind domain}
\subsection{Standard models}\label{Corti}
Corti \cite{CortiCubic} developed a theory of standard models of cubic surfaces over Dedekind domains. Let $C$ be a smooth projective connected curve and $p$ a point in $C$. Denote by  $\OO$ the spectrum of $\OO_{C, p}$ or $\whtOO_{C, p}$ and $K$ the quotient field of $\OO_{C, p}$ or $\whtOO_{C, p}$. Let $X_K$ be a cubic surface defined over $K$. A model of $X_K$ over $\OO$ is a flat projective family $X_{\OO}$ such that the generic fiber is $X_K$.

\begin{defn}
A standard model of $X_K$ over $\OO$ is a model $X_{\OO}$ over $\OO$ such that
\begin{enumerate}
\item $X_{\OO}$ has terminal singularities of index $1$.
\item The central fiber $X_0$ is reduced and irreducible.
\item The anticanonical system $-K_{X_\OO}$ is very ample and defines an embedding $X_\OO \subset \PP^3_\OO$.
\end{enumerate}
\end{defn}

The main theorem of \cite{CortiCubic} is the following.
\begin{thm}\cite{CortiCubic}
A standard model exists over $\OO$.
\end{thm}

Gluing local models together one gets a standard model over the curve $C$.

Note that in dimension $3$ terminal singularities of index $1$ are isolated and have multiplicity $2$.

For the singularities of the central fiber, we have the following, proved in \cite{WallCubic}.
\begin{lem}[ \cite{WallCubic}]
An integral cubic surface is either a cubic surface with du Val (=ADE) singularities, or a cone over an irreducible, possibly singular, plane cubic curve or a non-normal surface with only multiplicity $2$ singularities (along a line).
\end{lem}

To prove this lemma, one can look at the cases of singular cubic surfaces listed in \cite{WallCubic}. Note that an integral cubic surface is normal if and only if it has isolated singularities. In the list of \cite{WallCubic}, classes (A)-(C) corresponds to cubic surfaces with du Val singularities, class (D) the cone over a smooth plance cubic, class (E) the case of a non-normal cubic surface which is not a cone, and class (F) the cone over a nodal or cuspidal plane cubic.

In the following, we will analyze the local structure of the standard model over the formal neighborhood $\mcX \to \SP \CC[[t]]$ in the last $2$ cases.

The main result is the following.

\begin{prop}\label{prop:basechange}
Let $\mcX$ be a standard model over $\SP \CC[[t]]$ whose central fiber does not have du Val singularities. Then after a ramified base change $t=r^l$ and a birational modification, we get a new family $\mcX' \to \SP \CC[[r]]$ whose central fiber has du Val singularities. Furthermore, the Galois group $G \cong \ZZ/l\ZZ$ acts on the new total space and the projection $\mcX' \to \SP \CC[[r]]$ is $G$-equivariant.

Moreover, formal sections of the family $\mcX$ contained in the smooth locus induce $G$-equivariant formal sections of $\mcX' \to \SP\CC[[r]]$ contained in the smooth locus.

Finally, given two formal sections of the family $\mcX$ intersecting the central fiber in the smooth locus, let $x, y$ be the intersection points of the corresponding new $G$-equivariant sections with the new central fiber. Then $x, y$ are fixed points (under the $G$ action) in the smooth locus of the central fiber $\mcX'_0$ of the new family. Moreover, there is a $G$-equivariant map $f: \PP^1 \to \mcX'_0$ such that $f(0)=x, f(\infty)=y$ and the image of $f$ lies in the smooth locus of $\mcX'_0$. We may also assume that $f$ is very free.
\end{prop}

The proof of this proposition will be given in the following subsections. The first two consist of explicit computations of the base change and the correspondences between the sections. The last one proves the existence of equivariant very free curves in the smooth locus.

\subsection{Base change computation I: Cone over a plane cubic} 
\label{basechange1}

In the following two subsections, we will always denote the defining polynomial of the family as $H(t, X_0, X_1, X_2, X_3)$, which is a formal power series in $t$ with coefficients in $X_0, \ldots, X_3$. 

\begin{conv}
We say that a monomial $M(t, X_0, \ldots, X_3)$ is in the defining polynomial $H$ if after taking the power series expansion of $H$,  it appears as a monomial in $H$.
\end{conv}

Assume the central fiber is a cone over an irreducible plane cubic curve, defined by equation $F(X_1, X_2, X_3)=0$. Then the total space has multiplicity $3$ at vertex $[1, 0, 0, 0]$ unless we have $t X_0^3, t^2 X_0^3$, or $t X_0^2 X_i, i=1, 2, 3$ in the defining equation of the family. 

We first discuss how to find a ramified base change $t=r^n$ for some $n$ and a birational modification so that the central fiber has at worst du Val singularities.

Case (1): If $tX_0^3$ is contained in the defining polynomial $H$, we assign weight $(0, 1, 1, 1, 1)$ to $(X_0, X_1, X_2, X_3, r)$. Then the homogeneous polynomial $F$ has weight $3$. We make a degree $3$ base change $t=r^3$. There is exactly one more monomial with weight $3$ in $H(r^3, X_0, \ldots, X_3)$, namely $r^3X_0^3$ and the other monomials all have weight strictly greater than $3$. So after a degree $3$ base change $t=r^3$ and change of variables
\[
Y_0=X_0, rY_1= X_1, rY_2=X_2, rY_3=X_3,
\]
the new defining equation has the form
\[
r^3(Y_0^3+F(Y_1, Y_2, Y_3))+r^{\geq 4}R(Y_0, Y_1, Y_2, Y_3, r)=0,
\]
or equivalently,
\[
Y_0^3+F(Y_1, Y_2, Y_3)+r^{\geq 1}R(Y_0, Y_1, Y_2, Y_3, r)=0.
\]

The new family $\mcX' \to \SP \CC[[r]]$ has a $\ZZ/3 \ZZ$ action on the total space compatible with Galois group action on $\SP \CC[[r]]$. The central fiber of the new family is
\[
Y_0^3+F(Y_1, Y_2, Y_3)=0.
\]

Taking partial derivative with respect to $Y_0$ shows that the only possible singularities of the new central fiber lie in the plane $Y_0=0$ and come from singularities of the elliptic curve. Such singularities are isolated and have multiplicity $2$, and thus are du Val singularities. Furthermore, the singularities of $\mcX'$ lie in the curve $Y_0=F(Y_1, Y_2, Y_3)=0$.

Case (2): If $tX_0^3$ is not contained in the defining polynomial $H$, and $t X_0^2 X_i$ (for some $i=1, 2, 3$) is contained in $H$, we again assign weight $(0, 1, 1, 1, 1)$ to $(X_0, X_1, X_2, X_3, r)$. Then the homogeneous polynomial $F$ has weight $3$. One can make a degree $2$ base change $t=r^2$, and the change of variables
\[
Y_0=X_0, r Y_1= X_1, r Y_2=X_2, r Y_3=X_3.
\]
After a linear change of coordinates, one can write the equation for the new central fiber as
\[
Y_0^2Y_1+F'(Y_1, Y_2, Y_3)=0,
\]
where $Y_0=F'(Y_1, Y_2, Y_3)=0$ defines an irreducible plane cubic $C$. 

The singularities (if they exist) of the new central fiber are defined by equations
\[
Y_0Y_1=Y_0^2+\frac{\partial F'}{\partial Y_1}=\frac{\partial F'}{\partial Y_2}=\frac{\partial F'}{\partial Y_3}=0.
\]
They are of two kinds. One possible singularity lies in the plane $Y_0=0$ and comes from singularities of the plane cubic $C$. The other singularities lie in the plane $Y_1=0$. If $\frac{\partial F'}{\partial Y_1}$ vanishes at the singularities, then so does $Y_0$. Thus these singularities belong to the previous kind. If $\frac{\partial F'}{\partial Y_1}$ does not vanish at the singularities, then $Y_0$ is non-zero and $Y_1=0$ is tangent to the curve $C$ at a smooth point. There are two singularities of the second kind coming from two solutions of the equation $Y_0^2+\frac{\partial F'}{\partial Y_1}=0$, which satisfies the property $Y_0\neq 0$, and the two singular points are conjugate to each other under the $\ZZ/2\ZZ$ action. Again the new central fiber has du Val singularities only.

Case (3): If neither $t X_0^3$ nor $t X_0^2 X_i, i=1, 2, 3$ is contained in $H$, then $t^2X_0^3$ has to be contained in $H$. So one can make a degree $3$ base change $t=r^3$ and change of variables
\[
Y_0=X_0, r^2Y_1= X_1, r^2Y_2=X_2, r^2Y_3=X_3.
\]
Then the central fiber is
\[
Y_0^3+F(Y_1, Y_2, Y_3)=0,
\]
which has du Val singularities only.

Finally, a formal section $\whtt: \SP \CC[[t]] \to \mcX$ induces a $G$-equivariant formal section $\whtr: \SP \CC[[r]] \to \mcX'$. Let $[a, b, c, d]$ be the intersection of $\whtt$ with the central fiber. Assume it is a smooth point. So in particular, one of $b, c, d$ is non-zero. Then the induced section $\whtr$ intersects the new central fiber at $[0, b, c, d]$, which lies in the plane elliptic curve $C=\{Y_0=F(Y_1, Y_2, Y_3)=0\}$ or $\{Y_0=F'(Y_1, Y_2, Y_3)=0\}$. Moreover, the point $[0, b, c, d]$ is a smooth point of the elliptic curve, otherwise the point $[a, b, c, d]$ is a singular point of the surface $F(X_1, X_2, X_3)=0$ since it lies in the line spanned by the singular point in the curve $C$ and the vertex $[1, 0, 0, 0]$.

\subsection{Base change computation II: Non-normal and not a cone}\label{basechange2}
When the central fiber is non-normal but not a cone, by \cite{WallCubic}, p. 252, case E, the equation of the surface can be uniquely written as
\[
X_0X_2^2+X_1 X_3^2=0,
\]
or
\[
X_0 X_2^2+X_1 X_2 X_3 +X_3^3=0.
\]

In both cases the singular locus is the line $X_2=X_3=0$. Since the total space is smooth along the generic point of the line, we have a term $tF(X_0, X_1)$ in the defining polynomial $H$. 

In the first case, make a degree $2$ base change $t=r^2$ and change of variables
\[
Y_0= X_0, Y_1= X_1, rY_2=X_2, rY_3=X_3,
\]
one can get a new family $\mcX' \to \SP \CC[[r]]$, together with a $\ZZ/2\ZZ$ action on the total space compatible with action $r \mapsto -r$. The central fiber of the family is defined by
\[
F(Y_0, Y_1)+Y_0 Y_2^2+Y_1Y_3^2=0,
\]
which has only du Val singularities. One has a blow-up/blow-down description of this change of variables similar to the previous case. As in the the previous case, a formal section $\whtt$ induces a $G$-equivariant formal section $\wht{r}:\SP \CC[[r]] \to \mcX'$. If the original section intersects the central fiber in the smooth locus (i.e. one of the coordinates $X_2$ or $X_3$ is non-zero), then the new formal section $\wht{r}$ intersects the new central fiber in the line $Y_0=Y_1=0$, which lies in the smooth locus.

In the second case, we need to make different base changes and list them as follows.
\begin{enumerate}
\item If $tX_0^3$ is contained in the defining polynomial $H$, then make a degree $6$ base change $t=r^6$.
\item If $tX_0^3$ is not in $H$ and $tX_0^2X_1$ is in $H$, then make a degree $5$ base change $t=r^5$.
\item If neither $tX_0^3$ nor $tX_0^2X_1$ is in $H$ but $t X_0 X_1^2$ is in $H$, then make a degree $4$ base change $t=r^4$.
\item If none of $tX_0^3, tX_0^2X_1, t X_0 X_1^2$ are in $H$, and $tX_0^2X_3$ is in $H$, then make a degree $3$ base change $t=r^3$.
\item If none of $tX_0^3, tX_0^2X_1, t X_0 X_1^2$ are in $H$, but $tX_0^2X_3$ is in $H$, then make a degree $4$ base change $t=r^4$.
\end{enumerate}

After the base change, make the following change of variables
\[
Y_0=X_0, r Y_1=X_1, r^3 Y_2=X_2, r^2 Y_3=X_3.
\]
After the base change and change of variables, the central fiber has the form
\[
Y_0 Y_2^2+Y_1 Y_2 Y_3 +Y_3^3+G=0,
\]
where $G$ is one of the polynomials $Y_0^3, Y_0^2Y_1, a Y_0 Y_1^2+b Y_0^2Y_3, Y_1^3+c Y_0^2Y_2$.  This defines a cubic surface with at worst du Val singularities.  

A formal section $\whtt$ induces a new formal section $\wht{r}: \SP \CC[[r]] \to \mcX'$. If the original section $\wht{t}$ intersects the central fiber in the smooth locus (i.e. the coordinate $X_2$ is non-zero), then the new formal section $\wht{r}$ intersects the new central fiber at the point $[0, 0, 1, 0]$, which is a smooth point of the new central fiber.

\subsection{Equivariant curves}\label{equiv}

We first show the following.

\begin{lem}\label{lem:equiv}
Let $X$ be a smooth quasi-projective variety with an action of a finite cyclic group $G$ of order $l$, and let $T$ be an irreducible component of the fixed point loci of $G$. Assume that there is a very free curve thorough every point of $X$. Fix a $G$-action on $\PP^1$ by $[X_0, X_1] \mapsto [X_0, \zeta X_1]$, where $\zeta$ is a primitive $l$-th root of unity. Then given any two fixed points $x, y$ in $T$, there is a $G$-equivariant very free curve $f: \PP^1 \to X$ connecting $x$ and $y$.
\end{lem}

\begin{proof}
Given a point $x$ in $T$, the constant map $\PP^1 \to x$ is $G$-equivariant. By assumption, there is a very free rational curve in the smooth locus and passing through that point $x$. Then Lemma \ref{lem:EquivSmoothing}, (1), applied to $X$, shows that there is a very free $G$-equivariant rational curve mapping $0$ and $\infty$ to $x$. 

We can deform the curve with $0$ mapped to $x$ in a $G$-equivariant way to get a very free $G$-equivariant map connecting $x$ and a general point in $T$. To do this, first define two morphisms 
\[
s_0: 0 \times T \subset \PP^1 \times T \to X \times T
\]
\[
(0, t) \mapsto (x, t)
\]
and 
\[
s_\infty: \infty \times T \subset \PP^1 \times T \to X\times T
\]
\[
(\infty, t) \mapsto (t, t).
\]
Consider the relative Hom-scheme over $T$ fixing the two morphisms $s_0$ and $s_\infty$
\[
\text{Hom}_T(\PP^1 \times T, X \times T, s_0, s_\infty).
\]
There is a $G$-action on the relative Hom-scheme. The projection to $T$ is $G$-equivariant and smooth at the point represented by the very free curve mapping $0$ and $\infty$ to $x \in T$. So by Corollary \ref{sm-equiv-lifting}, the map from the Hom-scheme to $T$ is dominant and one can find such a deformation.

If $y$ is another point in $T$, the same construction gives a $G$-equivariant very free curve in the smooth locus connecting $y$ and a general point in $T$. To connect $x$ and $y$, take a common general point $z$ in $T$ and two very free $G$-equivariant rational curves connecting $x$ (resp. $y$) to $z$. Then part two of Lemma \ref{lem:EquivSmoothing} shows that there is a very free $G$-equivariant rational curve connecting $x$ and $y$. 
\end{proof}

By the description of the base change, and how the sections correspond to each other, the fixed points we need to connect in Proposition \ref{prop:basechange} are contained in the smooth locus of the central fiber and lie in a single irreducible component of the fixed point loci. Furthermore by \cite{XuSRC}, or by Theorem 21, \cite{BadReduction}, all cubic surfaces with at worst du Val singularities satisfy the condition that for any point in the smooth locus of the cubic surface, there is a very free rational curve in the smooth locus and passing through that point. Thus the last statement in Proposition \ref{prop:basechange} follows from the above lemma.

\section{Approximation in the smooth locus}

This section is devoted to a special case of weak approximation.

\subsection{Finding $G$-equivariant sections}
We first develop the techniques in \cite{HT06} in a $G$-equivariant setting. 
\begin{thm}\label{thm:G-equiv}
Let $G$ be a cyclic group of order $l$ and let $\mcX$ (resp. $C$) be a smooth proper variety (resp. a smooth projective curve) with a $G$-action. Let $\pi: \mcX \to C$ be a flat family of rationally connected varieties. Assume the following:
\begin{enumerate}
\item The morphism $\pi$ is $G$-equivariant. 
\item There is a $G$-equivariant section $s: C \to \mcX$.
\item The $G$-action on $C$ has a fixed point $p$ and the action of $G$ near $p$ is given by $t \mapsto \zeta t$, where $t$ is a local parameter and $\zeta$ is a primitive $l$-th root of unity.

\item The fiber of $\pi: \mcX \to C$ over the point $p$ is smooth. 
\end{enumerate}
Then for any positive integer $N$, and any $G$-equivariant formal section $\whts: \SP \whtOO_{p, C} \to \mcX$, there is a $G$-equivariant section $s'$ which agrees with the formal section $\whts$ to order $N$.
\end{thm}

The idea of the proof goes back to \cite{HT06}. Namely, we would like to add suitable rational curves to the given section and make $G$-equivariant deformations to produce a new section with prescribed jet data. The only subtlety in the proof is that in general we cannot choose the rational curves to be immersed. So instead of working with the normal sheaf as is done in \cite{HT06}, we work with the complex $\Omega_f$ defined as

\[
\begin{CD}
-1 & & 0 \\
f^*\Omega_X @> df^\dagger >> \Omega_C.
\end{CD}
\]
and its derived dual in the derived category. All the tensor products, duals, pull-backs, and push-forwards in the proof should also be taken as the derived functors in the derived category.

The following is a general form of the commonly used short exact sequences (of normal sheaves) which govern the deformation of a stable map from a nodal domain.

\begin{lem}\label{lem:def}
Let $f:C \cup D \to X$ be a morphism from a nodal curve $C\cup D$ with a single node to a smooth variety $X$ and $f_0$ (resp. $f_1$) the restriction of $f$ to $C$ (resp. $D$). Then
\begin{enumerate}
\item
We have the following distinguished triangles:
\[
\Omega_{f}^{\vee}\otimes \OO_D(-n) \to \Omega_f^{\vee} \to \Omega_f^{\vee} \otimes \OO_C \to \Omega_{f}^{\vee}\otimes \OO_D(-n)[1]
\]
\[
\Omega_{f_0}^{\vee} \to \Omega_f^{\vee}\otimes \OO_{C}  \to \epsilon[-1] \to \Omega_{f_0}^{\vee}[1]
\]
where $n$ is the preimage of the node in $D$, and $\epsilon$ is a skyscraper sheaf supported at the preimage of the node in $C$.
\item Let $G$ be a cyclic group of order $l$. Assume that there is a $G$-action on $C\cup D$ fixing each irreducible component. Then the node is a fixed point of the action and there is a natural $G$-action on all the complexes above. If locally around the node, the action is given by
\[
\begin{CD}
\CC[x, y]/xy @>>> \CC[x, y]/xy \\
(x, y)@>>> (\zeta x, \zeta^{-1} y),
\end{CD}
\]
where $\zeta$ is a primitive $l$-th roots of unity, then the $G$-action on $\epsilon$ is trivial.
\end{enumerate}
\end{lem}
\begin{proof}
The first distinguished triangle comes from restriction to the component $C$.

For the second distinguished triangle, consider the following distinguished triangles and the map between them:
\[
\begin{CD}
\Omega_{C\cup D} \otimes \OO_C @>>> \Omega_f \otimes \OO_C@>>> f^*\Omega_X \otimes \OO_C [1]@>>>\Omega_{C\cup D}\otimes \OO_C[1]\\
@VVV@VVV@|@VVV\\
\Omega_{C}  @>>> \Omega_{f_0} @>>> f_0^*\Omega_X[1] @>>>\Omega_{C}[1]
\end{CD}
\]
Therefore we have distinguished triangles
\[
\Omega_{C\cup D}\otimes \OO_C \to \Omega_C \to Q[1]\to \Omega_{C \cup D}\otimes \OO_C[1]
\]
\begin{equation}\label{eq:tri}
\Omega_f \otimes \OO_C \to \Omega_{f_0} \to Q'[1] \to \Omega_f \otimes \OO_C[1], 
\end{equation}
\[
Q[1] \to Q'[1] \to 0 \to Q[2].
\]
where $Q$ is a skyscraper sheaf supported at the node. The last distinguished triangle shows that $Q \cong Q'$. Taking dual of the distinguished triangle (\ref{eq:tri}) gives the second triangle in the lemma.

Part 2 of the lemma can be proved by a local computation. Or we can argue that the sheaf $\epsilon$ corresponds to a $G$-equivariant smoothing of the node. Therefore it has to be $G$-invariant.
\end{proof}

Now we begin the proof.

\begin{proof}[Proof of Theorem \ref{thm:G-equiv}] The proof is divided into two steps.

{\em Step 1: Approximation at $0$-th order.}

We may assume that $$H^1(C, \mcN_{C/\mcX}(-p))=0$$ by the same argument as in Lemma \ref{lem:EquivSmoothing}.

The section $s$ and the formal section $\whts$ intersect the fiber $\mcX_p$ at two fixed points of the $G$-action. Take a rational curve $D\cong \PP^1$ with a $G$-action as $[X_0, X_1] \mapsto [\zeta X_0, X_1]$, where $\zeta$ is the primitive $l$-th root of unity in the assumptions. By Theorem \ref{thm:equivariant}, there is a $G$-equivariant very free curve $ D\cong \PP^1 \to \mcX_p \to \mcX$ which maps $0=[1, 0]$ to $s(p)$ and $\infty=[0, 1]$ to $\whts(p)$. Let $f: C \cup D \to X$ be the nodal curve by combining the section and the curve $D$ and $f_0$ (resp. $f_1$) the restriction of $f$ to $C$ (resp. $D$).

By Lemma \ref{lem:def}, we have the following distinguished triangles:
\[
\Omega_{f}^{\vee}(-\infty) \otimes \OO_C(-p) \to \Omega_f{^{\vee}}(-\infty) \to \Omega_f{^{\vee}}(-\infty) \otimes \OO_D \to \Omega_{f}^{\vee}\otimes \OO_C(-p)[1]
\]
\[
\Omega_{f_0}^{\vee}(-p) \to \Omega_f^{\vee}\otimes \OO_{C}(-p)  \to \epsilon[-1] \to \Omega_{f_0}^{\vee}(-p)[1]
\]
\[
\Omega_{f_1}^{\vee}\otimes \OO_D(-\infty) \to \Omega_f^{\vee}\otimes \OO_{D}(-\infty)  \to \epsilon'[-1]\to \Omega_{f_1}^{\vee}\otimes \OO_D(-\infty)[1]
\]
where $\epsilon$ and $\epsilon'$ are torsion sheaves supported at the node of $C$ and $D$. Every complex has a natural $G$-action, and the $G$-actions on $\epsilon$ and $\epsilon'$ are trivial. Also note that $$\Omega_f(-\infty) \otimes \OO_C \cong \Omega_f \otimes \OO_C.$$

Taking hypercohomology gives long exact sequences
\begin{align}\label{eq:long1}
&0\to \HH^1(\Omega_{f}^{\vee}\otimes \OO_C(-p)) \to \HH^1(\Omega_f^{\vee}(-\infty)) \to \HH^1(\Omega_f^{\vee}(-\infty) \otimes \OO_D)\\
 \to &\HH^2(\Omega_{f}^{\vee}\otimes \OO_C(-p)) 
\to \HH^2(\Omega_f^{\vee}(-\infty)) \to \HH^2(\Omega_f^{\vee}(-\infty) \otimes \OO_D) \to \ldots,\nonumber
\end{align}

\begin{align}\label{eq:long2}
&0\to \HH^1(\Omega_{f_0}^{\vee}\otimes \OO_C(-p)) \to \HH^1(\Omega_f^{\vee}\otimes \OO_{C}(-p)) \to \epsilon\\ 
\to & \HH^2(\Omega_{f_0}^{\vee}\otimes \OO_C(-p)) \to \HH^2(\Omega_f^{\vee}\otimes \OO_{C}(-p))  \to 0,\nonumber
\end{align}
and
\begin{align}\label{eq:long3}
&0\to \HH^1(\Omega_{f_1}^{\vee}\otimes \OO_D(-\infty)) \to \HH^1(\Omega_f^{\vee}\otimes \OO_{D}(-\infty))  \to \epsilon'\\ 
\to & \HH^2(\Omega_{f_1}^{\vee}\otimes \OO_D(-\infty)) \to \HH^2(\Omega_f^{\vee}\otimes \OO_{D}(-\infty))  \to 0.\nonumber
\end{align}
 Note that $\Omega_{f_0}^{\vee}$ is quasi-isomorphic to $\mcN_{C/\mcX}[-1]$. Thus by the second long exact sequence,
\[
\HH^2(\Omega_{f_0}^\vee \otimes \OO_C(-p))=\HH^2(\Omega_f^{\vee} \otimes \OO_C(-p))=0.
\] 

Note that $\Omega_{f_1}^{\vee}$ is quasi-isomorphic to a shifted sheaf $\mcN[-1]$, where $\mcN$ is defined as the quotient in 
\[
0 \to T_D \to f^*T_X \to \mcN\cong f^*T_X/T_D \to 0.
\]
Since $f^*T_X$ is globally generated, $\HH^2(\Omega_{f_1}^{\vee}\otimes \OO_D(-\infty))=0$. Then by the third long exact sequence,
\[
\mathbb{H}^2(\Omega_{f}^{\vee}\otimes \OO_D(-\infty))=0.
\]

Therefore by the long exact sequence (\ref{eq:long1}),
\[
\HH^2(\Omega_f^{\vee}(-\infty))=0,
\]
and thus the $G$-equivariant deformation of the nodal curve $C\cup D$ with the point $\infty$ fixed is unobstructed.
 
Then by the long exact sequences (\ref{eq:long1}), (\ref{eq:long3}) and the vanishing, the composition of maps
\[
\HH^1(\Omega_f^{\vee}(-\infty))^G \to \HH^1(\Omega_f^{\vee}(-\infty) \otimes \OO_D)^G \to \epsilon'
\]
is surjective. Thus there is a $G$-equivariant deformation with $\infty$ fixed which smooths the node between $C$ and $D$.

{\em Step 2: Approximation at higher order.}

Assume that we have a section, still denoted by $s$, which agrees with $\whts$ to the $k (\geq 0)$-th order. We want to find a section agreeing with $\whts$ to order $k+1$.

Now let $\mcX_{k+1}$ be the $(k+1)$-th iterated blow-up of $\mcX$ associated to the formal section $\whts$. Then $G$ also acts on $\mcX_{k+1}$ and the projective to $C$ is $G$-equivariant. By abuse of notations, still denote the strict transforms of $s$ and $\whts$ by $s$ and $\whts$. Then they both intersect the exceptional divisor $E_{k+1} \cong \PP^d$ at fixed points of $G$. Assume the intersection points are different, otherwise there is nothing to prove.

Again we assume that $H^1(C, \mcN_{C/\mcX_{k+1}}(-p))=0$.

The key lemma is the following.

\begin{lem}\label{comb}
There is a comb $f: C\cup D \to \mcX_{k+1}$ from a nodal domain consisting of the given section $s(C)$ and suitable rational curves in the fiber such that
\begin{itemize}
\item $D=D_{k+1} \cup \cup_{j=1}^l R_j$, where $D_{k+1}\cong \PP^1$ and $R_j=\cup_{i=1}^k D_{ij}$ is a chain of rational curves. Denote by $x_j$ the node that connects $D_{k+1}$ to $R_j$.
\item There is a $G$-action on $D$ in the following way. The $G$-action on $D_{k+1}$ is given by 
\[
[X_0, X_1] \mapsto [X_0, \zeta^{-1} X_1].
\]
The group $G$ acts on $R_j, j=1, \ldots, l$ via a cyclic permutation among them. In particular, the points $x_j \in D_{k+1}$ are conjugate to each other under the $G$-action.

\item The morphism $f: C \cup D \to X$ is $G$-equivariant.

\item The $G$-fixed point $\infty=[0, 1]$ on $D_{k+1}$ is mapped to $\whts(p)$, and $0=[1, 0]$ on $D_{k+1}$ connects $C$.
\item The morphism $f: C \cup D$ is an immersion except at $0$ and $\infty$ in $D_{k+1}$.
\item The complex $\Omega_f^{\vee}$ satisfies the following vanishing conditions.

\begin{equation}\label{1}
\HH^2(\Omega_f^{\vee} \otimes \OO_C(-p))=\HH^2(\Omega_f^{\vee} \otimes \OO_{D_{k+1}}(-0-\infty))=\HH^2(\Omega_f^{\vee} \otimes \OO_{D_{i j}}(-1))=0,
\end{equation}

\begin{equation}\label{2}
\HH^2(\Omega_f^\vee(-\infty))=0,
\end{equation}
\begin{equation}\label{3}
\HH^2(\Omega_f^\vee\otimes \OO_{D_{k+1}}(-\infty-x_1-\ldots-x_l))^G=0.
\end{equation}

\end{itemize} 
\end{lem}

The construction is essentially the same as the one in \cite{HT06}, with the only difference coming from the consideration of the $G$-action. For an illustration of the comb $C\cup D$, see Figure. \ref{fig:comb} below and for the configuration of the comb with respect to the iterated blow-up $\mcX_{k+1}$, see Figure. \ref{fig:con}. 

\setlength{\unitlength}{0.8in}
\begin{figure}[h]
 \centering
 \includegraphics[width=4.4in]{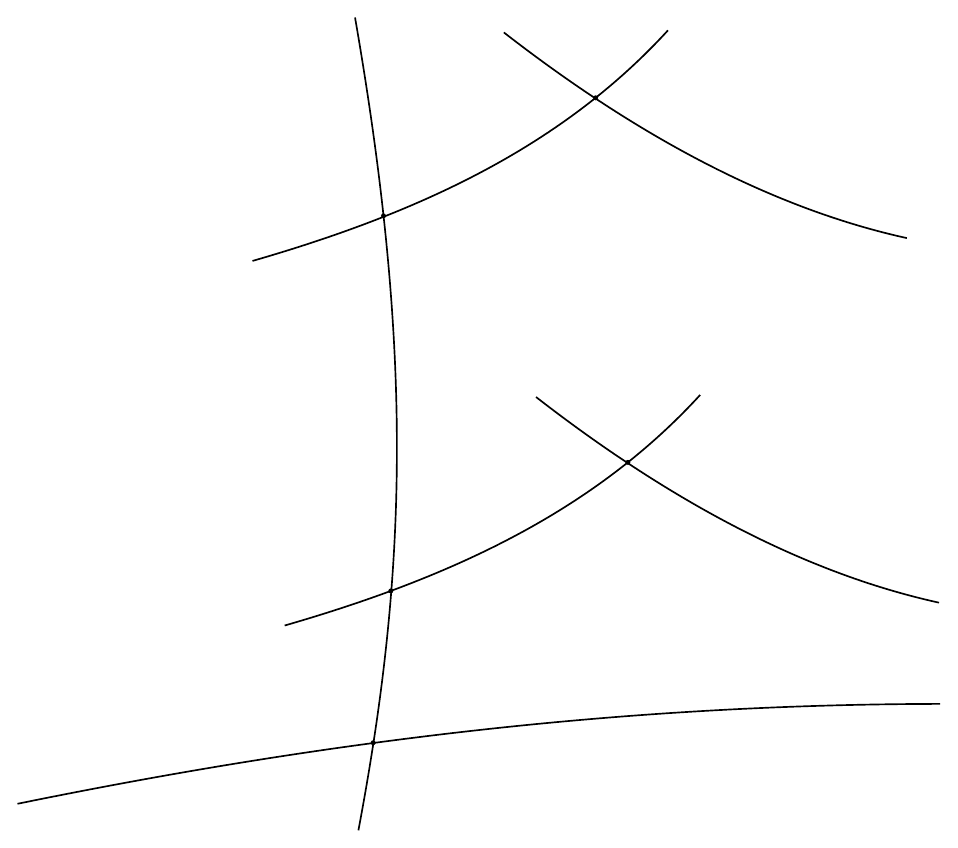}
 \put(-3.55,4.85){$D_{k+1}$}
 \put(-4.25,3.32){$R_1$}
 \put(-3.20,3.50){$x_1$}
 \put(-3.0, 2.50){$\dots\dots$}
 \put(-4.05,1.22){$R_l$}
 \put(-2.92,4.00){$D_{k1}$}
 \put(-3.18,1.30){$x_l$}
 \put(-1.52,4.00){$D_{k-1,1}$}
 \put(-2.72,1.90){$D_{kl}$}
 \put(-1.32,1.90){$D_{k-1,l}$}
 \put(-0.20,3.40){$\dots\dots$}
 \put(-0.05,1.35){$\dots\dots$}
 \put(-0.05,0.55){$C$}
 \caption{The comb $C\cup D$} 
 \label{fig:comb} 
\end{figure}

\begin{figure}[h]
 \centering
 \includegraphics[width=4.4in]{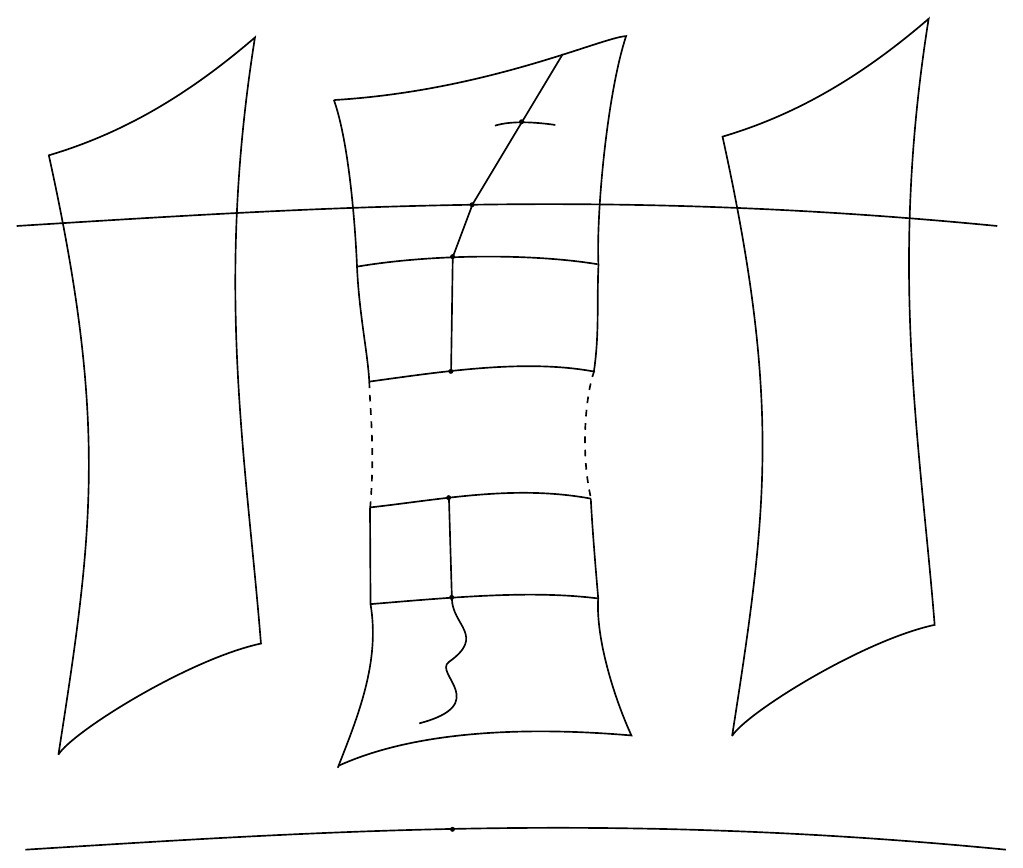}
 \put(-3.25,4.40){$E_{k+1}\cong\mathbb{P}^d$}
 \put(-3.60,3.75){$D_{k+1}\xrightarrow[]{l:1}L$}
 \put(-2.55,4.05){$\wht{s}(p)$}
 \put(-2.95,3.40){$S(p)$}
 \put(-3.25,3.12){$y_k$}
 \put(-3.45,2.70){$y_{k-1}$}
 \put(-3.05,3.10){$D_{k1}\dots D_{kl}$}
 \put(-3.05,2.95){are mapped}
\put(-3.05, 2.80){to this fiber}
 \put(-3.90,2.90){$E_k$}
 \put(-3.05,2.25){$\cdots$}
 \put(-3.25,1.85){$y_1$}
 \put(-3.05,1.83){$D_{11}\dots D_{1l}$}
\put(-3.05,1.68){are mapped}
\put(-3.05, 1.53){to this fiber}
 \put(-3.87,1.65){$E_1$}
 \put(-3.25,1.30){$y_0$}
 \put(-2.97,1.20){$D_{01}$}
 \put(-3.00,1.00){$\cdots$}
 \put(-3.05,0.80){$D_{0l}$}
 \put(-3.87,1.00){$E_0$}
 \put(-3.05,0.00){$P$}
 \put(-0.05,3.40){$S(C)$}
 \put(-0.05,0.00){$C$}
 \caption{Construction of the comb $C\cup D$}
\label{fig:con}
\end{figure}

\begin{proof}[Proof of Lemma \ref{comb}]
The line $L$ in $E_{k+1}\cong \PP^d$ joining $s(p)$ and $\whts(p)$ is invariant and intersects the exceptional divisor $E_{k}$ of $\mcX_{k+1}$ at a unique point $y_k$, which is necessarily a fixed point of $G$. Then there are $3$ fixed points in the line $L$ and thus all points are fixed points of $G$. Take a curve $D_{k+1}\cong \PP^1$. We impose a $G$-action on it by 
\[
[X_0, X_1] \mapsto [X_0, \zeta^{-1} X_1].
\]
Take an $l$-to-$1$ $G$-equivariant map from $D_{k+1}$ to the line $L$ such that $0=[1, 0]$ is mapped to $s(p)$ and $\infty=[0, 1]$ is mapped to $\whts(p)$. There are $l$ points $x_1, \ldots, x_l$, which lie in the same orbit of $G$, being mapped to the point $y_k \in E_k \cap E_{k+1}$, where $E_{k}$ and $E_{k+1}$ are exceptional divisors of the $(k+1)$-th iterated blow-up.

The exceptional divisor $E_k$ is isomorphic to the blow-up of $\PP^d$ at a point, thus is a $\PP^1$-bundle over $\PP^{d-1}$. Let $D_{k, 1}, \ldots, D_{k, l}$ be $l$ copies of $\PP^1$ each mapped isomorphically to the fiber curve $\PP^1$ containing the point $y_k$.

Inductively, let $y_{i}$ be the intersection point of $D_{i+1, 1}$ with $E_i$ and $D_{i, 1}, \ldots, D_{i, l}$ be $l$ copies of $\PP^1$ each mapped isomorphically to the fiber $\PP^1$ containing the point $y_i$ for all $i=k-1, \ldots, 1$. 

Finally let $y_0$ be the point of the intersection of $D_{1, 1}$ with the strict transform of $\mcX|_p$ and let $D_{0, 1} \ldots, D_{0, l}$ be $l$ copies of $\PP^1$ mapped to a very free curve in the strict transform of $\mcX_p$ intersecting $E_1$ at the point $y_0$. We may also assume that the maps are immersions.

Let $R_j$ be the chain of rational curves $\cup_{i=1}^{k}D_{i, j}$ connected to $D_{k+1}$ at the point $x_j$ for $j=1, \ldots, l$, and let $D$ be the curve $D_{k+1} \cup \cup_{j=1}^{l}R_{ j}$. There is a natural $G$-action on $D$, which permutes the $l$-chains of rational curves $R_j$ and acts on the irreducible component $D_{k+1}$ as specified above.

The restriction of the complex $\Omega_f^{\vee}$ to each curve $D_{i, j}$ is quasi-isomorphic to the normal sheaf with a shift $N_{f}[-1]$ (since the comb is an immersion along such curves). One can compute the restriction of $N_f$ to each curve $D_{i, j}$ as follows (see the proof of Sublemma 27, \cite{HT06}).
\begin{equation}\label{eqN}
\mcN_f|_{D_{i, j}}=\begin{cases}
\OO^{\oplus d}, & 1 \leq i \leq k,\\
\oplus_{n=1}^{d-1} \OO(a_n) \oplus \OO, a_n \geq 1 & i=0.
\end{cases}
\end{equation}

We now compute $\Omega_{f_{k+1}}^{\vee}$ on $D_{k+1}$, where $f_{k+1}$ is the restriction of the map to $D_{k+1}$ (i.e. the degree $l$ multiple cover of the line in $\PP^{d-1}$). This complex is quasi-isomorphic to the complex
\[
\begin{CD}
0 & & 1 \\
T_{D_{k+1}}\cong \OO(2) @>>> f_{k+1}^*T_{\mcX_{k+1}}\cong \OO(2l)\oplus \oplus_{i=1}^{d-2}\OO(l) \oplus \OO(-l),
\end{CD}
\]
Also note that the sheaf map $T_{D_{k+1}}\to f_{k+1}^*T_{\mcX_{k+1}}$ is injective and is the composition of maps $\OO(2) \to \OO(2l) \to f_{k+1}^*T_{\mcX_{k+1}}$.

We have a distinguished triangle
\begin{equation}\label{tricky}
\Omega_{f_{k+1}}^{\vee} \to \Omega_f^{\vee} \otimes \OO_{D_{k+1}} \to \epsilon[-1] \oplus \oplus_{j=1}^{l} \epsilon_j[-1]\to \Omega_{f_{k+1}}^{\vee}[1],
\end{equation}
where $\epsilon$ is a torsion sheaf supported at the node connecting $D_{k+1}$ and $C$, and $\epsilon_j$ is a torsion sheaf supported at the node connecting $D_{k+1}$ and $D_{k,j}$. The group $G$ acts on $\epsilon$ by the trivial action and acts on $\epsilon_j$ by permutation.

So the restriction of $\Omega_f^{\vee}$ to $D_{k+1}$ is quasi-isomorphic to the complex
\[
\begin{CD}
0 & & 1 \\
\OO(2) @>>>\OO(2l)\oplus \oplus_{i=1}^{d-2}\OO(l) \oplus \OO(1).
\end{CD}
\]
Since the above map maps the sheaf $\OO(2)$ injectively into the sheaf $\OO(2l)$, this complex is quasi-isomorphic to the shifted sheaf
\[
Q \oplus \oplus_{i=1}^{d-2}\OO(l) \oplus \OO(1) [-1],
\]
where $Q$ is the torsion sheaf defined as the quotient of $\OO(2) \to \OO(2l)$. Note that the $\OO(1)$ direction is the normal direction of the fiber.

Finally, the restriction of $\Omega_f^{\vee}$ to $C$ fits into the distinguished triangle
\[
\mcN_{C/\mcX_{k+1}}[-1] \to \Omega_f^\vee \otimes \OO_C \to \epsilon_0 \to \mcN_{C/\mcX_{k+1}},
\]
where $\epsilon_0$ is a torsion sheaf supported at the node.

Then the vanishing conditions (\ref{1}) are immediate from the identifications above. 

By the distinguished triangle
\[
\Omega_f^\vee \otimes \OO_C(-p) \to \Omega_f^\vee(-\infty) \to \Omega_f^\vee \otimes \OO_D(-\infty) \to \Omega_f^\vee \otimes \OO_C(-p)[1]
\]
and the three vanishing results in \ref{1}, we know that
\[
\HH^2(\Omega_f^\vee(-\infty))=0.
\]
This is the vanishing in (\ref{2}).

The vanishing in (\ref{3}) needs a little bit more work since it is only the $G$-invariant part of the hypercohomology group that vanishes.  First notice the following.
\begin{lem}\label{lem:smoothCD}
Assume only that the comb $C \cup D$ satisfies vanishing results (\ref{1}) and (\ref{2}). Then a general $G$-equivariant deformation of $C\cup D$ with $\infty$ fixed is unobstructed and smooths the node connecting $C$ and $D_{k+1}$.
\end{lem}
\begin{proof}
The vanishing result (\ref{2}) implies that the $G$-equivariant deformation of $C\cup D$ with $\infty$ fixed is unobstructed.

We first consider the following distinguished triangles
\begin{equation}\label{t2}
\Omega_f^\vee \otimes \OO_{D}(-\infty-0) \to \Omega_f^\vee(-\infty) \to \Omega_f^\vee(-\infty)\otimes \OO_C \to \Omega_f^\vee \otimes \OO_{D}(-\infty-0)[1]
\end{equation}
and
\begin{align} \label{t1}
&\oplus_{j=1}^l \Omega_f^\vee  \otimes \OO_{R_j}(-x_j) \to \Omega_f^\vee \otimes \OO_{D}(-\infty-0) \to \Omega_f^\vee \otimes\OO_{D_{k+1}}(-\infty-0) \\
\to &\oplus_{j=1}^l \Omega_f^\vee  \otimes \OO_{R_j}(-x_j)[1].\nonumber
\end{align}

Recall that $\Omega_f^\vee  \otimes \OO_{R_j}$ is quasi-isomorphic to a shifted normal sheaf $\mcN_f \otimes \OO_{R_j}[-1]$, and the sheaves $\mcN\otimes \OO_{R_j}$ are locally free and globally generated by (\ref{eqN}). Therefore
\[
\HH^2(\oplus_{j=1}^l \Omega_f^\vee  \otimes \OO_{R_j}(-x_j))=0,
\]
and thus by the distinguished triangle (\ref{t1}),
\[
\HH^2(\Omega_f^\vee \otimes \OO_{D}(-\infty-0))=0,
\]
which, combined with the long exact sequence of hypercohomology of the distinguished triangle (\ref{t2}), implies that the map
\begin{equation}\label{4}
\HH^1(\Omega_f^\vee(-\infty))^G \to \HH^1(\Omega_f^\vee(-\infty)\otimes \OO_C)^G
\end{equation}
is surjective.

Then we look at the distinguished triangle
\[
\Omega_{f_0}^\vee \to \Omega_f^\vee \otimes \OO_C \to \epsilon_0[-1] \to \Omega_{f_0}^\vee[1],
\]
where $f_0$ is the restriction of $f$ to $C$ and $\epsilon$ is a skyscraper sheaf supported at the point $p$.

By the vanishing results (\ref{1}), the map
\begin{equation}\label{6}
\HH^1(\Omega_f^\vee(-\infty) \otimes \OO_C)^G \to (\epsilon_0)^G=\epsilon_0
\end{equation}
is surjective. 

Note that $\Omega_f^\vee \otimes \OO_C \cong \Omega_f^\vee(-\infty)\otimes \OO_C$. Combining this identification and the surjectivity of maps in (\ref{6}) and (\ref{4}), we have proved that a general $G$-equivariant deformation with $\infty$ fixed smooths the node connecting $C$ and $D_{k+1}$.
\end{proof}

We have a distinguished triangle
\[
\Omega_{f_{k+1}}^\vee(-\infty) \to \Omega_f^\vee \otimes \OO_{D_{k+1}}(-\infty) \to \epsilon[-1] \oplus \oplus_{j=1}^{l} \epsilon_j[-1]\to \Omega_{f_{k+1}}^\vee(-\infty)[1],
\]
where $\epsilon$ is a torsion sheaf supported at $0 \in D_{k+1}$. This induces a map
\begin{equation}\label{8}
\HH^1(\Omega_f^\vee \otimes \OO_{D_{k+1}}(-\infty))^G \to \epsilon
\end{equation}
By Lemma \ref{lem:smoothCD}, a general deformation of $C\cup D$ with $\infty$ fixed is unobstructed and smooths the node connecting $C$ and $D_{k+1}$ (note that the proof of this result is independent of the vanishing (\ref{3})). Thus the composition
\[
\HH^1(\Omega_f^\vee(-\infty) \to \HH^1(\Omega_f^\vee \otimes \OO_{D_{k+1}}(-\infty))^G \to \epsilon
\]
is surjective. So the map in (\ref{8}) is also surjective. 

Recall that $\Omega_f^\vee \otimes \OO_{D_{k+1}}(-\infty)$ is quasi-isomorphic to the shifted sheaf
\[
(Q \oplus \oplus_{i=1}^{d-2}\OO(l) \oplus \OO(1))\otimes \OO_{D_{k+1}}(-\infty) [-1],
\]
and the $\OO(1)$ direction is the normal direction of the fiber.

Moreover the map in (\ref{tricky}) is can be written as
\begin{align*}
&Q \oplus \oplus_{i=1}^{d-2}\OO(l) \oplus \OO(-l) [-1] \to Q \oplus \oplus_{i=1}^{d-2}\OO(l) \oplus \OO(1) [-1]\\
\to  &\epsilon[-1] \oplus \oplus_{j=1}^{l} \epsilon_j[-1]
\to Q \oplus \oplus_{i=1}^{d-2}\OO(l) \oplus \OO(-l)
\end{align*}

Thus only the $\OO(1)\otimes \OO_{D_{k+1}}(-\infty)$
 summand may have a non-zero map to $\epsilon$ in the above evaluation map in (\ref{8}). Thus the unique section in this summand (i.e. the section of $H^0(\OO(1)\otimes \OO_{D_{k+1}}(-\infty))=H^0(\OO_{D_{k+1}})$) is mapped to a non-zero element in $\epsilon$. Furthermore, this unique section, thought of as a section in $$\HH^1(\Omega_f^\vee \otimes \OO_{D_{k+1}})^G$$
via the inclusion
\[
\HH^1(\Omega_f^\vee \otimes \OO_{D_{k+1}}(-\infty))^G \to \HH^1(\Omega_f^\vee \otimes \OO_{D_{k+1}})^G
\]
only vanishes at $\infty \in D_{k+1}$. Therefore the map
\begin{equation}\label{eq:surj}
H^0( \OO(1) \otimes \OO_{D_{k+1}}(-\infty))^G \to (\oplus_{j=1}^l \epsilon_j)^G
\end{equation}
is surjective.

To prove the vanishing in (\ref{3}), we only need to consider the $\OO(1)$ summand since all the other summands have enough positivity to kill the higher cohomology $\HH^2$. For the $\OO(1)$ summand, consider the short exact sequence
\[
0 \to \OO(1)\otimes \OO_{D_{k+1}}(-\infty-x_1-\ldots-x_l) \to \OO(1) \otimes \OO_{D_{k+1}}(-\infty) \to \oplus_{j=1}^l \epsilon_j \to 0,
\]
which induces a map on the $G$-invariant part of cohomology
\begin{align*}
&H^0( \OO(1) \otimes \OO_{D_{k+1}}(-\infty))^G \to (\oplus_{j=1}^l \epsilon_j)^G \\
\to &H^1(\OO(1)\otimes \OO_{D_{k+1}}(-\infty-x_1-\ldots-x_l))^G \to H^1( \OO(1) \otimes \OO_{D_{k+1}}(-\infty))^G.
\end{align*}
Since the map (\ref{eq:surj}) is surjective and $H^1( \OO(1) \otimes \OO_{D_{k+1}}(-\infty))^G$ vanishes, we have 
\[
H^1(\OO(1)\otimes \OO_{D_{k+1}}(-\infty-x_1-\ldots-x_l))^G=0,
\]
and thus
\[
\HH^2(\Omega_f^\vee \otimes \OO_{D_{k+1}}(-\infty-x_1-\ldots-x_l))^G=0.
\]
\end{proof}

We now finish the proof of step 2. Consider the distinguished triangles
\[
\Omega_{f}^{\vee}(-\infty) \otimes \OO_C(-p) \to \Omega_f^{\vee}(-\infty) \to \Omega_f^{\vee}(-\infty) \otimes \OO_D \to \Omega_{f}^{\vee}\otimes \OO_C(-p)[1]
\]
\begin{align*}
&\Omega_{f}^\vee\otimes \OO_{D_{k+1}}(-\infty-x_1-\ldots-x_l) \to \Omega_f^\vee \otimes \OO_{D}(-\infty)\\
 \to &\oplus_{j=1}^l \Omega_f^\vee \otimes \OO_{R_j} \to \Omega_{f}^\vee\otimes \OO_{D_{k+1}}(-\infty-x_1-\ldots-x_l)[1].
\end{align*}

The vanishings in (\ref{1}), (\ref{3}) imply that the map
\begin{equation}\label{7}
\HH^1(\Omega_f^\vee(-\infty))^G \to \HH^1(\Omega_f^\vee \otimes \OO_{D}(-\infty))^G \to \HH^1(\oplus_{j=1}^l \Omega_f^\vee \otimes \OO_{R_j})^G
\end{equation}
is surjective (note that $\Omega_f^\vee \otimes \OO_{R_j}\cong \Omega_f^\vee(-\infty) \otimes \OO_{R_j}$).

Since the $G$-action on the chain of rational curves $R_j$ is permutation. There is a section of
\[
\HH^1(\oplus_{j=1}^l \Omega_f^\vee \otimes \OO_{R_j})^G
\]
which is mapped to a non-zero element in the $G$-invariant part of the torsion sheaf supported at the nodes on $R_j, j=1, \ldots, l$ if and only if there is a section of 
\[
\HH^1( \Omega_f^\vee \otimes \OO_{R_j})
\]
which is mapped to a non-zero element in the torsion sheaf supported at the nodes on $R_j$ and for some (and hence for all) $j$. 

Since the restriction of $\Omega_f^\vee$ to $R_j$ is quasi-isomorphic to $\mcN_f\otimes \OO_{R_j}[-1]$ and $\mcN_f\otimes \OO_{R_j}$ is locally free and globally generated by the vanishing (\ref{1}) or (\ref{eqN}), this follows from the same argument as in \cite{HT06} (in particular, the bottom of P. 187 and P. 188).

So combining this observation with the surjectivity of the map in (\ref{7}) and Lemma \ref{lem:smoothCD}, we have proved that a general $G$-equivariant deformation with $\infty$ fixed smooths all the nodes and produces a new section which agrees with $\whts$ to order $k+1$.
\end{proof}

\subsection{Weak approximation in the smooth locus}
The following is a special case of weak approximation, which turns out to be all one needs to finish the proof. The basic idea is that when the central fiber is very singular, a base change and a birational modification will greatly improve the singularities. Then one just need to keep track of the Galois group action to get back to the original family.
\begin{lem}\label{lem:approximate}
Let $\pi: \mcX \to B$ be a standard model of families of cubic surfaces over a smooth projective curve $B$ and $s: B \to \mcX$ be a section. Let $b_1, \ldots, b_k, b_{k+1}, \ldots, b_m$ be finitely many points in $B$, and $\whts_j$, $k+1 \leq j \leq m$ be formal sections over the points $b_j, k+1\leq j \leq m$, which lie in the smooth locus of $\pi: \mcX \to B$. Assume that the section $s$ intersects the fibers over $b_{k+1}, \ldots, b_m$ in the smooth locus. Then given a positive integer $N$, there is a section $s': B \to \mcX$ such that $s'$ is congruent to $s$ modulo $\mfm_{B,b_i}^{N}$ for all $1 \leq i \leq k$, and congruent to $\whts_j$ modulo $\mfm_{B, b_j}^N$, for all $k+1 \leq j \leq m$.
\end{lem}
\begin{proof}
By the iterated blow-up construction, keeping the jet data is the same as keeping section intersect certain exceptional divisors in the iterated blow-up, which, in turn, is equivalent to keeping the intersection numbers of the section with exceptional divisors in the iterated blow-up. In the following we will only use deformation/smoothing argument and we will only use a general deformation (i.e. without specialization) to prove the weak approximation result. Thus the intersection numbers are always kept. Since we start with a section which intersects the fibers over given points in the smooth locus, the section we produce by adding curves in the smooth locus of $\pi$ and smoothing also intersect the fibers over the given points in the smooth locus. Therefore we can reduce the general case to the case that $b_{k+1}, \ldots, b_m$ is just a single point $b\in B$.

First of all we may approximate the formal section if the fiber over $b$ has at worst du Val singularities. Indeed the section $s$ and the formal section we need to approximate lie in the smooth locus over $b$ by assumption. So the statement follows from Theorem 1, \cite{BadReduction} and the fact that the smooth locus of log del Pezzo surfaces are strongly rationally connected \cite{XuSRC} (or Theorem 21, \cite{BadReduction} for the case of cubic surfaces). 

From now on we assume the fiber over the point $b$ is either a cone over a plane cubic curve or non-normal. By Proposition \ref{prop:basechange}, at least for the formal neighborhood, we can find a ramified base change and a birational modification so that the new central fibers have du Val singularities only and the Galois group acts on the total space of the formal neighborhood.

The next goal is to show that we can make the base change globally on the curve $B$.

Given finitely many points $x_1, x_2, \ldots, x_n$ in $B$, and any positive integer $l$, there is a cyclic cover of degree $l$ of $B$ which is totally ramified over $x_1, \ldots, x_n$ (and other points). To see this, take a general Lefschetz pencil which maps $x_1, \ldots, x_n$ (and other points) to $0 \in \PP^1$ and is unramified over these points. Take a degree $l$ map $B_1=\PP^1 \to \PP^1, [X_0, X_1] \mapsto [X_0^l, X_1^l]$ and let $C=B\times_{\PP^1} B_1$ be the fiber product. Then $C$ is the desired cyclic cover. We may also choose the cover $C \to B$ so that the preimages of $b_1, \ldots, b_k$ are $l$ distinct points.

Let $C \to B$ be a cyclic cover of degree $l$ (which is determined in subsections \ref{basechange1} and \ref{basechange2} according to the type of singularities), totally ramified at the point $b$ (and other points). There is a new family over the curve $C$ by base change. The cyclic group $G=\ZZ/l \ZZ$ acts on $C$ and the total space of the new family over $C$ in such a way that the projection to $C$ is $G$-equivariant. Let $c$ be the points in $C$ which is mapped to $b$. One can modify the family locally around $c$ as in subsections \ref{basechange1} and \ref{basechange2}. 

Let $\mcX' \to C$ be the family after the base change and birational modifications. The group $G$ still acts on the total space $\mcX'$ and the projection to $C$ is $G$-equivariant.

The section $s$ induces a $G$-equivariant section of the new family $\mcX' \to C$ and has the desired jet data at all points mapped to $b_1, \ldots, b_k$. Still denote the new section by $s$. Moreover, the new $G$-equivariant section $s$ intersects the fiber over the point $c$ in the smooth locus (Proposition \ref{prop:basechange}). So do the new formal sections we want to approximate.

Now the argument in Theorem \ref{thm:G-equiv} proves weak approximation in this case. The theorem needs the assumption that the fiber over $c$ is smooth. But this can be weakened as the following:
\begin{itemize}
\item the section and the formal sections intersect the fibers over $c$ in the smooth locus, and
\item there are $G$-equivariant very free curves in the smooth locus connecting the intersection points of the central fiber with the $G$-equivariant section and formal sections over the points $c$.
\end{itemize}
The second condition is proved in Proposition \ref{prop:basechange}.
\end{proof}

\section{Proof of the main theorem}
We first show that there are ``nice" sections for a standard model of a families of cubic surfaces. The idea goes back to an argument of Keel-$\text{M}^{\text{c}}\text{Kernan}$ \cite{KeelMckernan} Sec. 5, in particular, the proof of Corollary 5.6. Hassett-Tschinkel also used the idea of Keel-$\text{M}^{\text{c}}\text{Kernan}$ to study strong rational connectedness in \cite{WAhypersurface}, which is very similar to the argument presented here.
\begin{lem}\label{lem:smoothlocus}
Let $\pi: \mcX \to B$ be a standard model of family of cubic surfaces over a smooth projective curve $B$ and $s: B \to \mcX$ be a section. Given finitely many points $b_1, \ldots, b_k$ in $B$, and a positive integer $N$, there is a section $s': B \to \mcX$ such that $s'$ is congruent to $s$ modulo $\mfm_{B,b_i}^{N}$ and $s'(B-\cup b_i)$ lies in the smooth locus of $\pi: \mcX \to B$.
\end{lem}
\begin{proof}
One first resolves the singularities of $\mcX$ along the fibers over $b_i$ in such a way that the partial resolution is an isomorphism except along these fibers. Then use the iterated blow-up construction according to the jet data of $s$ near the points $b_i$. After sufficiently many iterated blow-ups, fixing the jet data is the same as passing through fixed components. Call the new space $\mcX_1$.

Then the lemma is reduced to showing that there is a section of the new family $\mcX_1 \to B$ which has desired intersection number with irreducible components of the fibers over $b_1, \ldots, b_k$ in $B$ and lies in the smooth locus of $\mcX_1 \to B$. 

In the following proof, we will show that the given section $s$, after adding very free rational curves in general fibers, deforms away from the singular locus of the total space. Since the deformation will not change the intersection numbers with divisors, we get a deformation into the smooth locus with the jet data fixed.

Take a resolution of singularities $\mcX_2 \to \mcX_1$ which is an isomorphism over the smooth locus such that the exceptional locus in $\mcX_2$ consists of simple normal crossing divisors $E_i, i=1, \ldots, n$. After adding very free curves in general fibers and smoothing, we may assume the strict transform of the section $s$, denoted by $f: B \to \mcX_2$, passes through $g+1$ very general points $p_1, \ldots, p_{g+1}$ in $\mcX_2$, where $g$ is the genus of $B$.

First consider the Kontsevich moduli space of stable maps $\overline{M}_{g, g+1}(\mcX_1)$ parameterizing stable maps from genus $g$ curves with $g+1$ marked points to $\mcX_1$, which pass through $g+1$ very general points $p_1, \ldots, p_{g+1}$ in $\mcX_1$. Let $V$ be an irreducible component containing the point represented by the map $f: B \to \mcX_1$. 

Next consider the Kontsevich moduli space of stable maps $\overline{M}_{g, g+1}(\mcX_2)$ parameterizing stable maps from genus $g$ curves with $g+1$ marked points to $\mcX_2$, which pass through $p_1, \ldots, p_{g+1}$. There is a natural forgetful map from $\overline{M}_{g, g+1}(\mcX_2)$ to the corresponding moduli space of stable maps to $\mcX_1$. Take $U$ to be the inverse image of $V$ and write the restriction of the forgetful map as $F: U \to V$.

Note that $U$ has at most countably many irreducible components. Clearly the forgetful map $F$ surjects onto an open dense subset of $V$ since we can always lift a section from $\mcX_1$ to $\mcX_2$, with all the conditions still satisfied. Thus there is an irreducible component $U_0$ of $U$ which dominates $V$ (here we are using the fact that $\CC$ is uncountable). By the following lemma (to be proved later), $\dim U_0=-K_{\mcX_2} \cdot D-2(g+1)$, where $f': D \to \mcX_2$ is general point (hence $D$ is irreducible) in $U_0$.

\begin{lem}\label{lem:dim}
Let $X$ be a smooth $3$-fold. Assume that given $g+1$ very general points $x_1, \ldots, x_{g+1}$ in $X$, there is an embedding $f: C \to X$ of a smooth projective curve of genus $g$ in $X$ which maps $g+1$ points $c_1, \ldots, c_{g+1}$ in $C$ to $x_1, \ldots, x_{g+1}$. Let $f': (D, d_1, \ldots, d_{g+1}) \to (X, x_1, \ldots, x_{g+1})$ be a general deformation of $C$. Then $H^1(D, \mcN_{D/X}(-d_1-\ldots-d_{g+1}))=0$. In particular, every irreducible component containing the morphism 
\[
f: (C, c_1, \ldots c_{g+1}) \to (X, x_1, \ldots, x_{g+1})
\]
 has dimension $-K_X \cdot C-2(g+1)$.
\end{lem}

The standard model has isolated cDu Val singularities, i.e. $3$-fold terminal and local complete intersection singularities. So does the new total space $\mcX_1$ by construction. Therefore every irreducible component containing the point $f': D \to \mcX_2 \to \mcX_1$ has dimension at least $-K_{\mcX_1}-2(g+1)$ since $\mcX_1$ has local complete intersection singularities. Furthermore, by definition of terminal singularities, we have
\[
-K_{\mcX_1}=-K_{\mcX_2}+\sum_{i=1}^n a_i E_i, a_i >0.
\]
Thus if the image of $D$ in $\mcX_1$ intersects the singular locus, $-K_{\mcX_1} \cdot D$ is strictly larger than $-K_{\mcX_2}$, which is impossible.

Therefore we have a section $s': B\to \mcX_1$ which has the desired intersection numbers and lies in the smooth locus of the total space $\mcX_1$. Finally note that if a section lies in the smooth locus of the total space $\mcX_1$, then the section lies in the smooth locus of the morphism $\pi: \mcX_1 \to B$.
\end{proof}

\begin{proof}[Proof of the lemma \ref{lem:dim}]
Let $M$ be an irreducible component of the Kontsevich moduli space of genus $g$ stable maps with $g+1$ marked points to $X$ containing $f: C \to X$. Then the evaluation map $$ev: M  \to \underbrace{X \times \ldots \times X}_{g+1}$$ is dominant (here we use the fact that $\CC$ is uncountable). Let $$f': (D, d_1, \ldots, d_{g+1}) \to X$$ be a general point in the moduli space $M$. Then $D$ is also embedded and one can fix $g (\geq 0)$ general points in the curve $D$ and deform the curve along the normal direction at a general point. This implies that we have an exact sequence of sheaves:
\[
H^0(D, \mcN_{D/X}(-d_1-\ldots-d_{g}))\otimes \OO_D \to \mcN_{D/X}(-d_1-\ldots-d_{g}) \to Q \to 0,
\]
where $Q$ is a torsion sheaf on $D$ and $d_1, \ldots, d_g$ are general points in $D$. It follows from the exact sequence that $H^1(D, \mcN_{D/X})=0$ since a general degree $g$ line bundle has no $H^1$.

We also have short exact sequences
\begin{align*}
&0 \to \mcN_{D/X}(-d_1-\ldots-d_{k+1}) \to \mcN_{D/X}(-d_1-\ldots-d_k)\\
 \to &\mcN_{D/X}(-d_1-\ldots-d_k)|_{d_{k+1}} \to 0,
\end{align*}
for all $ k=0, \ldots, g$.

Again if the points $d_i$ are general, then the maps on global sections
\[
H^0(D, \mcN_{D/X}(-d_1-\ldots-d_k)) \to \mcN_{D/X}(-d_1-\ldots-d_k)\vert_{d_{k+1}}, 0 \leq k \leq g
\]
are surjective. Thus we have
\begin{align*}
&H^1(D, \mcN_{D/X})=H^1(D, \mcN_{D/X}(-d_1))=\ldots\\
=&H^1(D, \mcN_{D/X}(-d_1-\ldots-d_{g+1})=0.
\end{align*}
\end{proof}

Now we have all the results needed for the proof of Theorem \ref{main}.

\begin{proof}[Proof of Theorem \ref{main}]

Given a smooth cubic hypersurface $X$ over the function field $\CC(B)$ of a smooth projective curve $B$, one can find a Lefschetz pencil over $\PP^1_{\CC(B)}$ whose general fiber is a smooth cubic hypersurface of one dimension lower. If weak approximation holds for general fibers, then weak approximation holds for the total family (c.f. Theorem 3.1, \cite{WASurvey}). Thus it suffices to prove weak approximation for all smooth cubic surfaces.

Let $\mcX \to B$ be a standard model for a smooth cubic surface defined over the function field $\CC(B)$.

By Lemma \ref{lem:smoothlocus}, one may choose a section $s$ of $\pi: \mcX \to B$ which lies in the smooth locus of $\pi$. Given a finite number of formal sections $\whts_i, 1 \leq i \leq m$ over $b_i \in B, 1 \leq i \leq m$, let $\wht{L}_i$ be the line in the formal neighborhood $\mcX_{\whtOO_{B, b_i}}$ which joins the formal sections $\whts_i$ and $s$. Let $\whts'_i$ be the third intersection points of $\wht{L}_i$ with the formal neighborhood. Up to perturbing the formal sections $\whts_i$, we may assume the three local formal sections are not the same.

Choose an integer $N$ large enough. We claim that there is a line $L$ defined over the field $\CC(B)$, which contains the rational point $p$ corresponding to $s$, intersects the generic fiber $\mcX_\eta$ at a cycle of degree $3$, and agrees with $\wht{L}_i$ to order $N$. This is equivalent to weak approximation for the space of lines through the rational point $p$. Since the space is isomorphic to $\PP^2$, weak approximation holds.

If the line $L$ intersects the generic fiber at two other rational points, then we connect the two sections corresponding to the two rational points with rational curves in general fibers and smooth them with the jet data fixed.

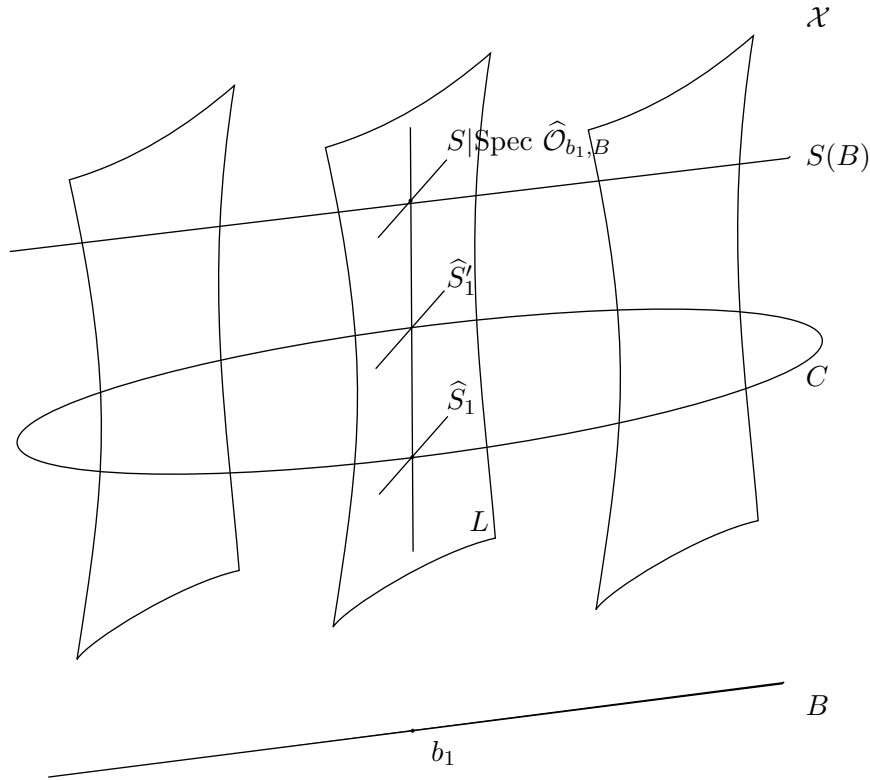
\begin{figure}[h]
 \centering
 \includegraphics[width=4.4in]{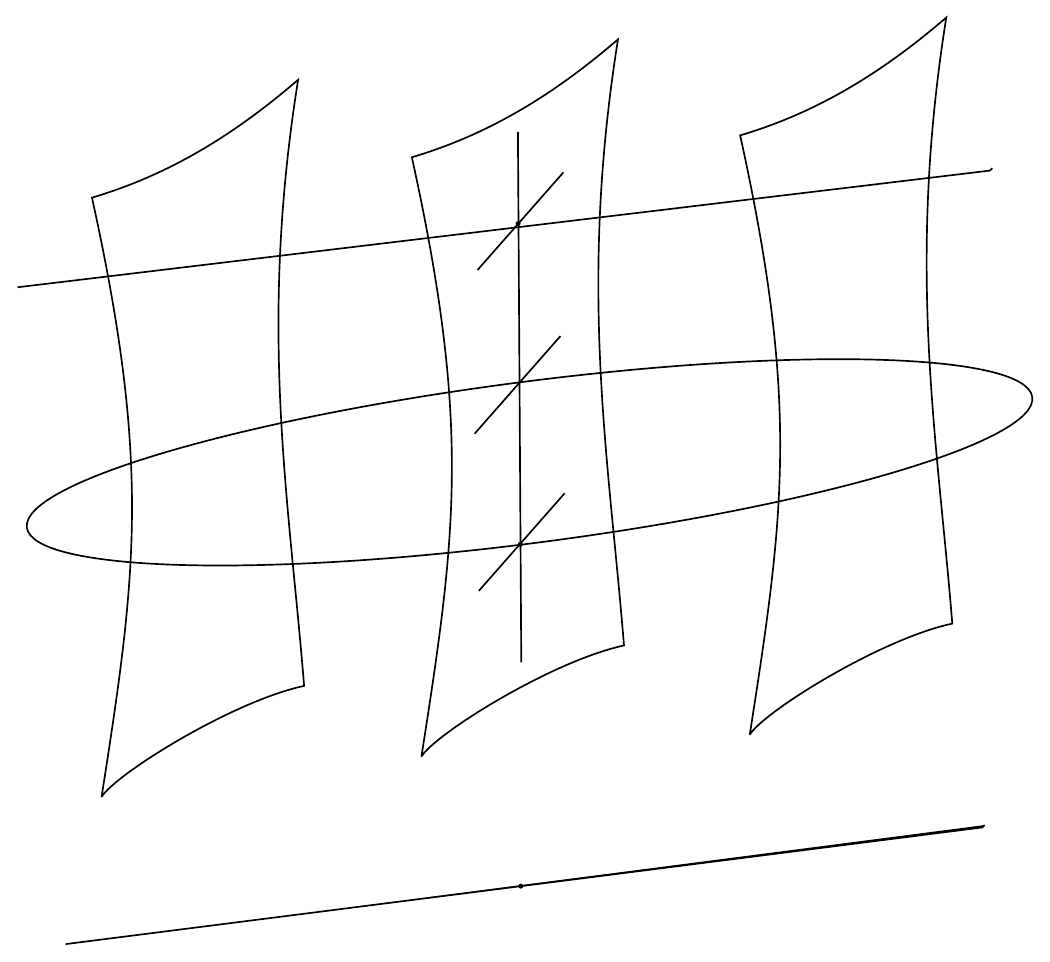}
 \put(-0.20,5.00){$\mathcal{X}$}
 \put(-2.55,4.20){$S|\SP \whtOO_{b_1, B}$}
 \put(-2.55,3.30){$\wht{S}_1'$}
 \put(-2.55,2.50){$\wht{S}_1$}
\put(-2.40, 1.70){$L$}
 \put(-2.65,0.20){$b_1$}
 \put(-0.20,4.10){$S(B)$}
 \put(-0.20,2.65){$C$}
 \put(-0.20,0.50){$B$}
 \caption{Producing the multisection $C$} 
 \label{fig:C} 
\end{figure}

So we may assume that there is a degree $2$ multisection $C \to \mcX$ such that the formal sections induced by $C$ agree with $\whts_i$ and $\whts'_i$ to order $N$. See Figure \ref{fig:C} for an illustration of the situation (we draw the section $s$ and its restriction to the formal neighborhood differently so that it is easier to visualize).

Make the base change $C \to B$, which is et{\'a}le over the points $b_i$. Denote the preimages of $b_i$ to be $c_i$ and $c'_i$. The formal sections $\whts_i$ and $\whts'_i$ induce formal sections over $c_i$ and $c_i'$, which will still be denoted by $\whts_i$ and $\whts'_i$. The section $s$ also induces a section of the new family, still denoted by $s$. The degree $2$ multi-section $C$ induces another section $s_C$ of the new family, which agrees with $\whts_i'$ to order $N$.

We may arrange the degree $2$ cover $C \to B$ to be \'etale over all the points whose fibers are singular. So the \'etale neighborhood of singular fibers of the new family over $C$ is isomorphic to the \'etale neighborhood of the corresponding singular fibers over $B$. As a consequence, the family over $C$ is a standard model over $C$ in the sense of \ref{Corti}.

By Lemma \ref{lem:smoothlocus}, there is a section $\tilde{s}_C$ of the family over $C$, which agrees with $\whts_i'$ to order $N$ and otherwise lies in the smooth locus of the fibration. Then by Lemma \ref{lem:approximate}, there is a section $\sigma_C$ of the new family which agrees with both $\whts'_i$ and the restriction of $s$ to the formal fibers over $c_i$ to order $N$. See Figure \ref{fig:WAC} for an illustration of the situation.

\begin{figure}[h]
  \centering
 \includegraphics[width=4.4in]{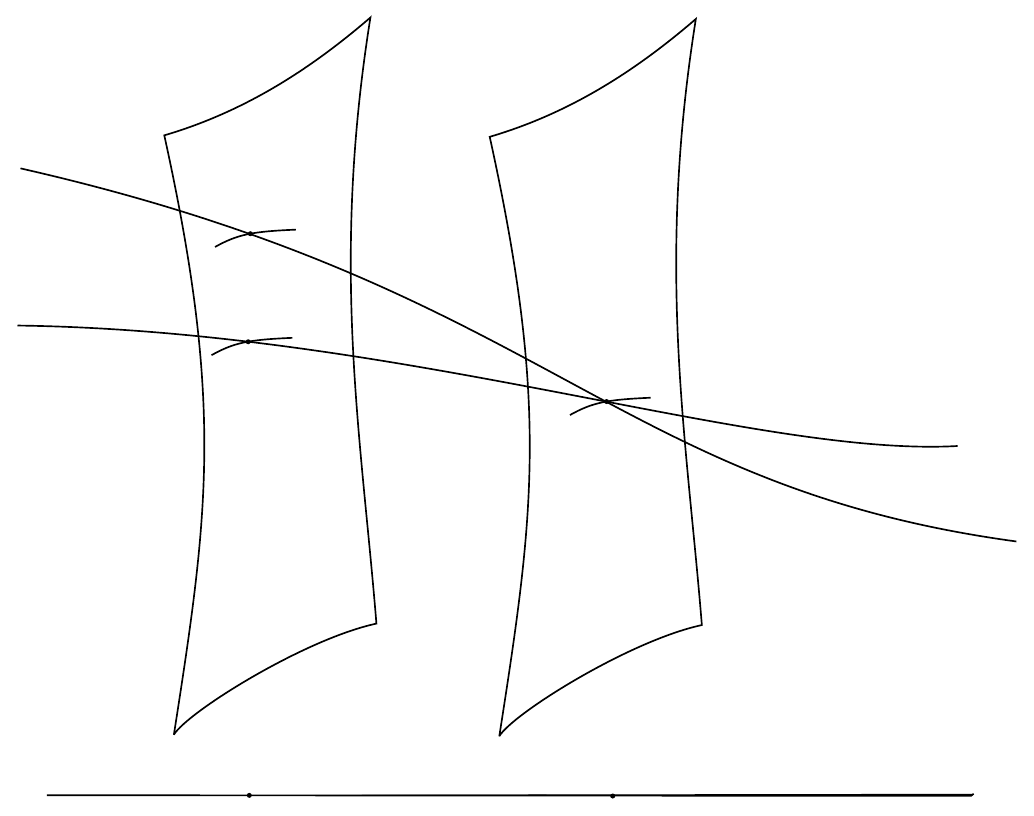}
 \put(-2.25,2.30){$\wht{S}_1'$}
 \put(-4.05,3.20){$S|\SP \whtOO_{b_1, B}$}
 \put(-4.05,2.60){$\wht{S}_1$}
 \put(-0.3,1.90){$S_C(C)$}
 \put(-0,1.40){$\sigma_C(C)$}
 \put(-0,0){$C$}
 \put(-4.15,-.10){$C_1$}
 \put(-2.25,-.10){$C_1'$}
 \caption{Weak approximation for the family over $C$} 
 \label{fig:WAC} 

\end{figure}
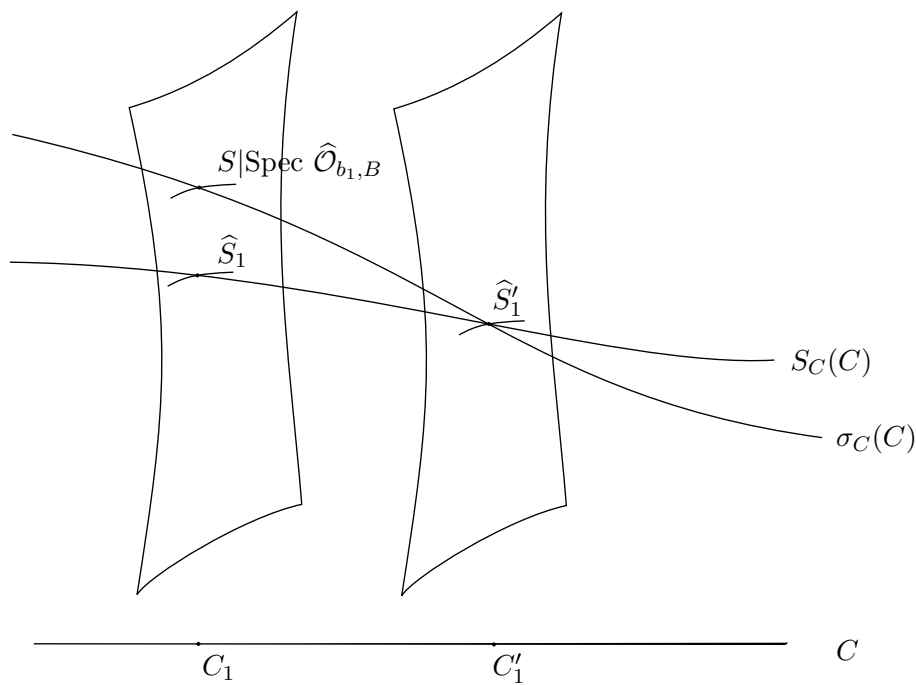

The section $\sigma_C$ gives a degree $2$ multisection, still denoted by $\sigma_C(C)$, of the original family $\pi: \mcX \to B$, which agrees with $s$ and $\whts_i'$ to order $N$ in the formal neighborhood of the points $b_i$. Take the family of lines spanned by the degree $2$ multisection $\sigma_C(C)$. This family corresponds to a line $\tilde{L}$ defined over the generic fiber, i.e. the function field $\CC(B)$. The line $\tilde{L}$ agrees with the line $L$ to order $N$ over the points $b_i$ by construction. We have a third intersection point of $\tilde{L}$ with the cubic surface, necessarily defined over $\CC(B)$. The section corresponding to this rational point will agree with the formal sections $\whts_i$ to order $N$ by construction, thus completing the proof (c.f. Figure \ref{fig:final}).

\begin{figure}[h]
 \centering
 \includegraphics[width=4.4in]{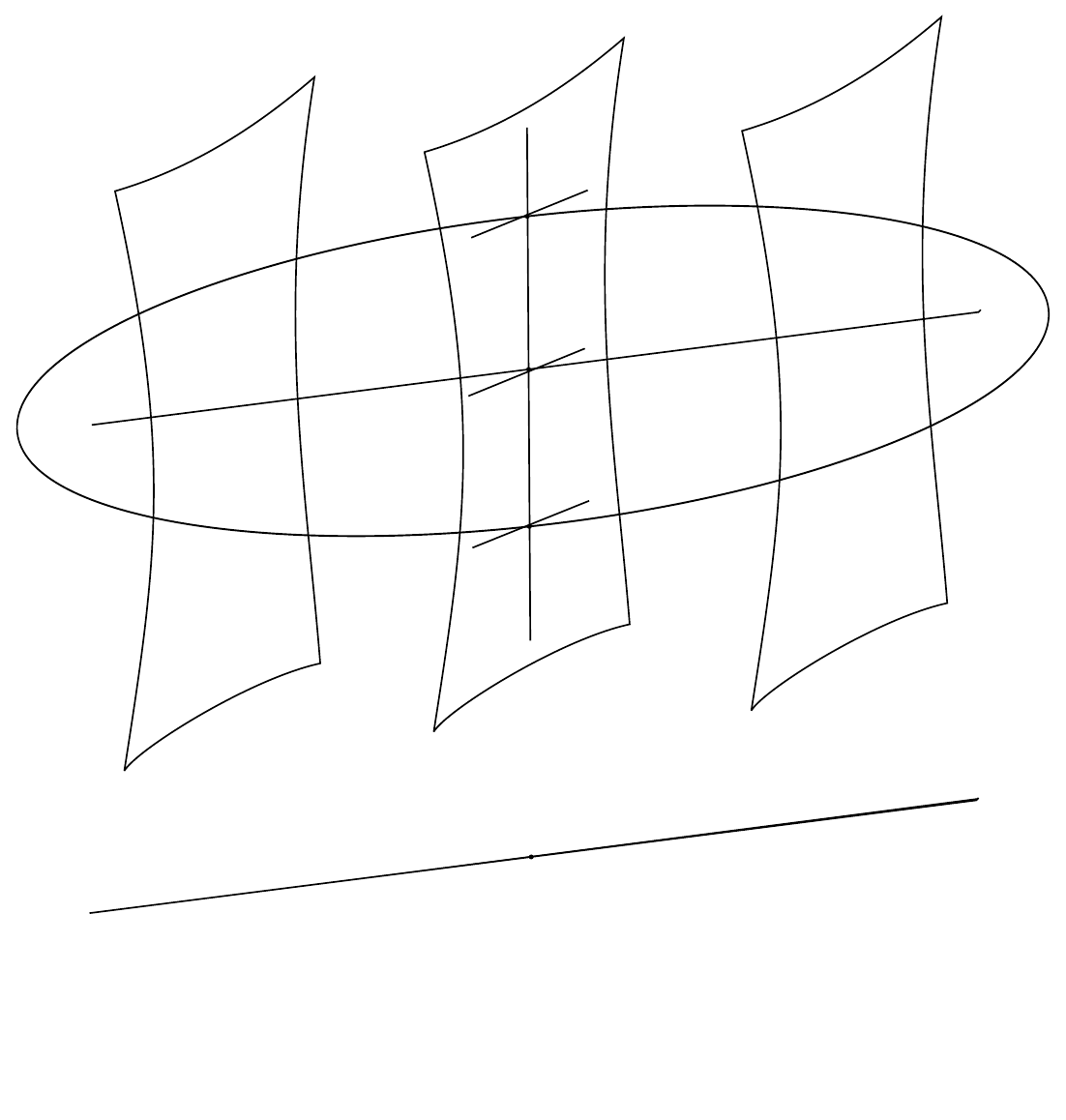}
 \put(-2.55,4.82){$S|\SP \whtOO_{b_1, B}$}
 \put(-2.55,4.02){$\wht{S}_1$}
 \put(-2.55,3.22){$\wht{S}_1'$}
 \put(-2.40, 2.50){$\tilde{L}$}
 \put(-0.55,3.40){$\sigma_C(C)$}
 \put(-0.55,1.40){$B$}
 \caption{Producing the section with desired jet data.} 
 \label{fig:final} 
\end{figure}

\end{proof}


\end{document}